\documentclass[reqno]{amsart}

\usepackage{amsmath,amssymb,amscd, accents} \usepackage{graphicx}
\DeclareGraphicsExtensions{.eps} \usepackage{mathrsfs}
\usepackage[mathcal]{eucal} 
\usepackage{esint}

\newtheorem{Thm}{Theorem}{\bfseries}{\itshape}
\newtheorem*{Thm*}{Theorem}{\bfseries}{\itshape}
\newtheorem{Cor}{Corollary}{\bfseries}{\itshape}
\newtheorem{Prop}[Cor]{Proposition}{\bfseries}{\itshape}
\newtheorem{Lem}[Cor]{Lemma}{\bfseries}{\itshape}
\newtheorem*{Lem*}{Lemma}{\bfseries}{\itshape}
\newtheorem{Fact}[Cor]{Fact}{\bfseries}{\itshape}
{\bfseries}{\itshape}
\newtheorem{Def}[Cor]{Definition}{\bfseries}{\rmfamily}
\newtheorem{Ex}[Cor]{Example}{\scshape}{\rmfamily}
\newtheorem{Rem}[Cor]{Remark}{\scshape}{\rmfamily}
{\bfseries}{\itshape}

\renewcommand\ge{\geqslant} \renewcommand\le{\leqslant}
\let\tildeaccent=\~ \let\hataccent=\^
\renewcommand\~[1]{\widetilde{#1}}

\def\<{\left<} \def\>{\right>} \def\({\left(} \def\){\right)}

\def\abs#1{\left\vert #1 \right\vert} \def\norm#1{\left\Vert #1
  \right\Vert} \def\size#1{\mathbf S\left(#1 \right)}

\let\parasymbol=\S \def\secref#1{\parasymbol\ref{#1}}

\let\polishL=l \def\Zoladek.{\.Zol\c adek}

\def\Re{\operatorname{Re}} \def\Im{\operatorname{Im}}

\def\etc.{\emph{etc}.}
 \def\Sing{\operatorname{Sing}}

\def\:{\colon} \def\R{{\mathbb R}} \def\C{{\mathbb C}} \def\Z{{\mathbb
    Z}} \def\N{{\mathbb N}} \def\Q{{\mathbb Q}}

\let\PolishL=\L 
\def\L{{\mathbb L}}

 \def\e{\varepsilon} \def\S{\varSigma}
\def\l{\lambda}   
 
\def\poly{{\operatorname{poly}}}

 \def\d{\,\mathrm d}

 \def\Lojas.{\PolishL ojasiewicz}
\def\cN{{\mathcal N}} \def\cH{{\mathcal H}}
\def\cP{{\mathcal P}}

\def\cF{{\mathcal F}}  
\def\cI{{\mathcal I}}  
 \def\cD{{\mathcal D}}
\def\cM{{\mathcal M}}
\def\cO{{\mathcal O}}

\def\cV{{\mathcal V}}

\def\rest#1{{\vert_{#1}}}

\def\vol{\operatorname{Vol}}

\def\Gal{\operatorname{Gal}}

\def\w{\omega}

\def\id{\operatorname{id}}

\def\Qa{\Q^{\mathrm{alg}}}

\def\alg{\mathrm{alg}}
\def\trans{\mathrm{trans}}
\def\size{\mathrm{size}}
\def\RE{\mathrm{RE}}
\def\den{\operatorname{den}}

\def\vf{{\mathbf f}}
\def\vg{{\mathbf g}}
\def\vx{{\mathbf x}}

\def\vz{{\mathbf z}}
\def\vw{{\mathbf w}}
\def\vp{{\mathbf p}}
\def\va{{\mathbf a}}
\def\vc{{\mathbf c}}
\def\vk{{\mathbf k}}
\def\vchi{{\boldsymbol\chi}}

\begin{document}

\title{Wilkie's conjecture for restricted elementary functions}

\author{Gal Binyamini} 
\address{Weizmann Institute of Science, Rehovot, Israel}
\email{gal.binyamini@weizmann.ac.il}

\author{Dmitry Novikov}
\address{Weizmann Institute of Science, Rehovot, Israel}
\email{dmitry.novikov@weizmann.ac.il}

\begin{abstract}
  We consider the structure $\R^\RE$ obtained from $(\R,<,+,\cdot)$ by
  adjoining the restricted exponential and sine functions. We prove
  Wilkie's conjecture for sets definable in this structure: the number
  of rational points of height $H$ in the transcendental part of any
  definable set is bounded by a polynomial in $\log H$. We also prove
  two refined conjectures due to Pila concerning the density of
  algebraic points from a fixed number field, or with a fixed
  algebraic degree, for $\R^\RE$-definable sets.
\end{abstract}
\date{\today}
\maketitle

\section{Introduction}

\subsection{Statement of the main results}
Our main object of study is the structure
\begin{equation}
  \R^\RE = (\R,<,+,\cdot,\exp\rest{[0,1]},\sin\rest{[0,\pi]}).
\end{equation}
The superscript $\RE$ stands for ``restricted elementary''. We
consider the natural language for $\R^\RE$, where we also include
constants for each real number. We will refer to formulas in this
language as $\R^\RE$-formulas.

For a set $A\subset\R^m$ we define the \emph{algebraic part} $A^\alg$
of $A$ to be the union of all connected semialgebraic subsets of $A$
of positive dimension. We define the \emph{transcendental part}
$A^\trans$ of $A$ to be $A\setminus A^\alg$.

Recall that the \emph{height} of a (reduced) rational number
$\tfrac a b\in\Q$ is defined to be $\max(|a|,|b|)$. More generally,
for $\alpha\in\Qa$ we denote by $H(\alpha)$ its absolute
multiplicative height as defined in \cite{bombieri:heights}. For a
vector $\boldsymbol\alpha$ of algebraic numbers we denote by
$H({\boldsymbol\alpha})$ the maximum among the heights of the coordinates.

Let $\cF\subset\C$ denote a number field which will be fixed
throughout this paper. For a set $A\subset\C^m$ we denote the set of
$\cF$-points of $A$ by $A(\cF):=A\cap\cF^m$ and denote
\begin{equation}
  A(\cF,H) := \{\vx\in A(\cF):H(\vx)\le H\}.
\end{equation}
The following is our main result.
\begin{Thm}\label{thm:main-F}
  Let $A\subset\R^m$ be $\R^\RE$-definable. Then there exist integers
  $\kappa:=\kappa(A)$ and $N=N(A,[\cF:\Q])$ such that
  \begin{equation}
    \#A^\trans(\cF,H) \le N\cdot(\log H)^\kappa.
  \end{equation}
\end{Thm}

Theorem~\ref{thm:main-F} establishes a conjecture of Wilkie
\cite[Conjecture~1.11]{pila-wilkie} for the case of the restricted
exponential function. It also establishes a refined version due to
Pila \cite[Conjecture~1.4]{pila:exp-alg-surface}, who conjectured that
the exponent $\kappa$ can be chosen to be independent of the field
$\cF$. For a statement of the full conjectures and an outline of the
history of the problem see~\secref{sec:background}.

We also prove an additional conjecture of Pila
\cite[Conjecture~1.5]{pila:exp-alg-surface} (in the case of the
restricted exponential) on counting algebraic points of a fixed degree
without restricting to a fixed number field. For $k\in\N$ we denote
\begin{align}
  A(k) &:= \{\vx\in A:[\Q(x_1):\Q],\ldots,[\Q(x_m):\Q]\le k\}, \\
  A(k,H) &:= \{\vx\in A(k):H(\vx)\le H\}.  
\end{align}
Then we have the following.
\begin{Thm}\label{thm:main-k}
  Let $A\subset\R^m$ be $\R^\RE$-definable. Then there exist
  integers $\kappa:=\kappa(A,k)$ and $N=N(A,k)$ such that
  \begin{equation}
    \#A^\trans(k,H) \le N\cdot(\log H)^\kappa.
  \end{equation}
\end{Thm}
Note that in Theorem~\ref{thm:main-k} the exponent $\kappa$ may depend
on the degree $k$.

\subsection{Background}
\label{sec:background}

In \cite{bombieri-pila}, Bombieri and Pila considered the following
problem: let $f:[0,1]\to\R$ be an analytic function and $X\subset\R^2$
its graph. What can be said about the number of integer points in $X$,
and more generally in the homothetic dilation $tX$? They showed that
if $f$ is transcendental then for every $\e>0$ there
exists a constant $c(f,\e)$ such that $\#(tX\cap\Z^2)\le c(f,\e)t^\e$
for all $t\ge1$. The condition of transcendence is necessary, as can
be observed by the simple example $f(x)=x^2$ satisfying
$\#(tX\cap\Z^2)\simeq t^{1/2}$. The proof of \cite{bombieri-pila}
introduced a new method of counting integer points using certain
interpolation determinants.

In \cite{pila:density-Q} Pila extended the method of
\cite{bombieri-pila} to the problem of counting rational points on
$X$. In particular, he proved that if $f$ is transcendental then for
every $\e>0$ there exists a constant $c(f,\e)$ such that
$\#X(\Q,H)\le c(f,\e)H^\e$ for all $H\in\N$. In this generality, the
asymptotic $O(H^\e)$ is essentially the best possible, as illustrated
by \cite[Example~7.5]{pila:subanalytic-dilation}.

Moving beyond the case of curves one encounters a new phenomenon: a
set $X$ may be transcendental while still containing algebraic curves,
and in such a case (as illustrated by the graph of $x\to x^2$) one
cannot expect the asymptotic $\#X(\Q,H)=O(H^\e)$. However, in
\cite[Theorem~1.1]{pila:subanalytic} Pila showed that for compact
subanalytic surfaces this is the only obstruction. More precisely, for
any compact subanalytic surface $X\subset\R^n$ and $\e>0$ there exists
a constant $C(X,\e)$ such that $\#X^\trans(\Q,H)\le C(X,\e)H^\e$. The
same result for arbitrary compact subanalytic sets was conjectured in
\cite[Conjecture~1.2]{pila:subanalytic-dilation}. In
\cite[Theorem~1.8]{pila-wilkie} Pila and Wilkie proved this conjecture
in a considerably more general setting. Namely, they showed that if
the set $X$ is definable in \emph{any} O-minimal structure and $\e>0$
then there exists a constant $C(X,\e)$ such that
$\#X^\trans(\Q,H)\le C(X,\e)H^\e$. This result contains in particular
the case of compact subanalytic sets (and more generally globally
subanalytic sets), obtained for the O-minimal structure of restricted
analytic functions $\R_{\mathrm {an}}$. It also contains much wider
classes of definable sets, for instance those definable in the
structure $\R_{{\mathrm{an}},\exp}$ obtained by adjoining the graph of
the \emph{unrestricted} exponential function to $\R_{\mathrm{an}}$.
The Pila-Wilkie theorem in this generality turned out to have many
important diophantine applications (see e.g. \cite{scanlon:survey} for
a survey).

As mentioned earlier, the asymptotic $O(H^\e)$ is essentially the best
possible if one allows arbitrary subanalytic sets (even analytic
curves). However, one may hope that in more tame geometric contexts
much better estimates can be obtained. In
\cite[Conjecture~1.11]{pila-wilkie} Wilkie conjectured that if $X$ is
definable in $\R_{\exp}$, i.e. using the unrestricted exponential but
without allowing arbitrary restricted analytic functions, then there
exist constants $N(X)$ and $\kappa(X)$ such that
\begin{equation}
  \#X^\trans(\Q,H)\le N(X)\cdot(\log H)^{\kappa(X)}.
\end{equation}
In \cite[Conjectures~1.4~and~1.5]{pila:exp-alg-surface} Pila proposed two
generalizations of this conjecture: namely, that for an arbitrary
number field $\cF\subset\R$ one has
\begin{equation}
  \#X^\trans(\cF,H)\le N(X,\cF)\cdot(\log H)^{\kappa(X)}
\end{equation}
where only the constant $N(X,\cF)$ is allowed to depend on $\cF$, and
for $k\in\N$ one has
\begin{equation}
  \#X^\trans(k,H)\le N(X,k)\cdot(\log H)^{\kappa(X,k)}
\end{equation}
where both constants are allowed to depend on $k$.

Some low-dimensional cases of the Wilkie conjecture have been
established. In \cite[Theorem~1.3]{pila:pfaff} Pila proved the analog
of the Wilkie conjecture for graphs of Pfaffian functions
(see~\secref{sec:pfaffian} for the definition) or plane curves defined
by the vanishing of a Pfaffian function. In
\cite[Corollary~5.5]{jones-thomas} Jones and Thomas have shown that
the analog of the Wilkie conjecture holds for surfaces definable in
the structure of restricted Pfaffian functions. In
\cite{pila:exp-alg-surface,butler} the Wilkie conjecture is confirmed
for some special surfaces defined using the \emph{unrestricted}
exponential.

\subsection{Overview of the proof}

\subsubsection{The Pfaffian category}
Our approach is based on an interplay between ideas of complex
analytic geometry, and the theory of Pfaffian function. We briefly
pause to comment on the latter. In \cite{khovanskii:fewnomials}
Khovanskii introduced the class of Pfaffian functions, defined as
functions satisfying a type of triangular system of polynomial
differential equations (see~\secref{sec:pfaffian} for details). The
Pfaffian functions enjoy good finiteness properties, and have played a
fundamental role in Wilkie's work on the model-completeness of
$\R_{\exp}$ \cite{wilkie:Rexp}.

From the Pfaffian functions one can form the class of semi-Pfaffian
sets, i.e. sets defined by a boolean combination of Pfaffian
equalities and inequalities, and sub-Pfaffian sets, i.e. projections
of semi-Pfaffian sets. Pfaffian functions have a natural notion of
degree, and by works of Khovanskii \cite{khovanskii:fewnomials} and
Gabrielov and Vorobjov \cite{gv:complexity,gv:compact-approx} the
number of connected components of any semi- or sub-Pfaffian set can be
explicitly estimated from above in terms of the degrees of the
Pfaffian functions involved (see
Theorem~\ref{thm:pfaffian-complexity}). Moreover, for us it important
that these estimates are \emph{polynomial} in the degrees.

\subsubsection{The holomorphic-Pfaffian category}
The theory of Pfaffian functions is an essentially real theory, based
on topological ideas going back to the classical Rolle theorem. The
holomorphic continuation of a real Pfaffian function on $\R^n$ is not,
in general, a Pfaffian function on $\C^n\simeq\R^{2n}$. However, since
our arguments are complex-analytic in nature we restrict attention to
\emph{holomorphic-Pfaffian} functions: holomorphic functions whose
graphs (in an appropriate domain) are sub-Pfaffian sets. It is a small
miracle that the graph of the complex exponential $e^z$, and hence
also of $\sin z$, is indeed a Pfaffian set when restricted to a strip.
A similar feature is used in an essential way in the work of van den
Dries \cite{vdd:Rre}, which we discuss below.

\subsubsection{The inductive scheme for counting rational points}
Fix some domain $\Omega\subset\C^n$. We begin by explaining our
estimate for $\#X(\Q,H)$ when $X\subset\C^n$ is a
\emph{holomorphic-Pfaffian variety}, i.e. a set cut out by
holomorphic-Pfaffian equations of Pfaffian degree $\beta$. For
simplicity we assume $X=X^\trans$. Our basic strategy is similar to
the strategy used by Pila and Wilkie \cite{pila-wilkie}: we seek to
cover $X$ by smaller pieces $X_k$, such that for each piece one can
find an algebraic hypersurface $H_k$ with
$X_k(\Q,H)\subset X_k\cap H_k$.

By way of comparison, in \cite{pila-wilkie} the subdivision is
performed in two steps. One first applies a reparametrization theorem
to write $X$ as the union of images of $C^r$-smooth maps
$\phi_j:(0,1)^{\dim X}\to X$ with unit norms: this step is independent
of $H$. One then subdivides each cube $(0,1)^{\dim X}$ into $H^\e$
subcubes, and for each subcube constructs the hypersurface as above
using a generalization of the Bombieri-Pila method
\cite{bombieri-pila}. Crucially for \cite{pila-wilkie}, the degrees of
the these hypersurfaces can be chosen to depend only on $\e$ but not
on $H$. However, to go beyond the asymptotic $O(H^\e)$ it appears that
one must allow the degrees to depend on $H$.

In our approach $X$ is covered by $\poly(\beta)$ pieces
$X_k:=X\cap\Delta_k$ where $\Delta_k$ is a \emph{Weierstrass polydisc}
for $X$, a notion introduced below (more accurately we take $\Delta_k$
to be a Weierstrass polydisc shrunk by a factor of two). We then
construct a hypersurface $H_k$ of degree $\poly(\beta,\log H)$
containing $X_k(\Q,H)$. Consequently we replace $X$ by
$X\cap(\cup_k H_k)$, which is guaranteed to have strictly smaller
dimension and degree polynomial in $\beta$ and $\log H$. One can then
finish the proof by induction on dimension, and eventually obtain a
zero-dimensional holomorphic-Pfaffian variety defined by equations of
degree $\poly(\log H)$ and hence having at most $\poly(\log H)$
points. We proceed to explain the subdivision step and the
construction of the algebraic hypersurfaces.

\subsubsection{Weierstrass polydiscs and holomorphic decompositions}
We define a \emph{Weierstrass polydisc}
$\Delta=\Delta_z\times\Delta_w$ for $X$ to be a polydisc in some
coordinate system, where $\dim X=\dim\Delta_z$ and
$X\cap(\Delta_z\times\partial\Delta_w)=\emptyset$. It follows from
this definition that the projection from $X\cap\Delta$ to $\Delta_z$
is a finite (ramified) covering map, and all fibers have the same
number of points (counted with multiplicities). We denote this number
by $e(X,\Delta)$ and call it the \emph{multiplicity} of $\Delta$.
Weierstrass polydiscs are ubiquitous in complex analytic geometry: they
are the basic sets where a complex analytic variety can be expressed
as a finite cover of a polydisc.

By a variant of Weierstrass division we prove the following polynomial
interpolation result: for any holomorphic function $f$ on a
neighborhood of $\Delta$ there is function $P$ on $\Delta$,
holomorphic in the $z$-variables and polynomial of degree at most
$e(X,\Delta)$ in each of the $w$-variables, such that $f\equiv P$ on
$X\cap\Delta$ (see Proposition~\ref{prop:weierstrass-interpolation}).
Moreover, the norm of $P$ can be estimated in terms of the norm of
$f$. The existence of such a decomposition immediately implies that
$\Delta$ is the domain of a \emph{decomposition datum} in the sense of
\cite{me:interpolation} (see Definition~\ref{def:decomposition}). Then
the results of \cite{me:interpolation}, themselves a complex-analytic
analog of the Bombieri-Pila interpolation determinant method
\cite{bombieri-pila}, imply that $(X\cap\Delta)(\Q,H)$ is contained in
an algebraic hypersurface of degree $d=\poly(e(X,\Delta),\log H)$ (for
a precise statement see Proposition~\ref{prop:hypersurface-select}).
It will therefore suffice to cover $X$ by $\poly(\beta)$ Weierstrass
polydiscs $\Delta$ each satisfying $e(X,\Delta)=\poly(\beta)$.

\subsubsection{Covering by Weierstrass polydiscs}
The multiplicity $e(X,\Delta)$ is relatively easy to estimate using
Pfaffian methods, being the number of isolated solutions of a system
of Pfaffian equations and inequalities. The heart of the argument is
therefore the covering by Weierstrass polydiscs. For this purpose we
prove the following, somewhat stronger statement (see
Theorem~\ref{thm:pfaff-polydisc}): if $B\subset\Omega$ is ball of
radius $r$ around a point $p\in\Omega$, then there is a Weierstrass
polydisc $\Delta\subset B$ for $X$ with center $p$ and polyradius at
least $r/\poly(\beta)$. In other words, every point $p$ is the center
of a relatively large Weierstrass polydisc. From this it is easy to
deduce that $X$ can be covered by $\poly(\beta)$ Weierstrass polydiscs.

We briefly comment on the proof of Theorem~\ref{thm:pfaff-polydisc}.
Suppose first that $X$ has complex codimension $1$. In this case we
show, by a simple geometric argument, that the theorem can be reduced
to finding a ball of radius $r/\poly(\beta)$ disjoint from
$S^1\cdot X$, where $S^1=\{|\zeta|=1\}$ acts by scalar multiplication
on $\C^n$. The set $S^1\cdot X$ is also sub-Pfaffian, now of real
codimension $1$. We use an argument involving metric entropy,
specifically Vitushkin's formula (in the form given by Friedland and
Yomdin \cite{yomdin-friedland}) to show that $S^1\cdot X$ can be
covered by relatively few balls of radius $r/\poly(\beta)$, and
elementary considerations then show that it must be disjoint from one
(in fact, many) such ball. For $X$ of arbitrary codimension we use an
induction on codimension by repeated projections.

\subsubsection{From $\R^\RE$-definable sets to holomorphic-Pfaffian varieties}
At this point our review of the proof for holomorphic-Pfaffian $X$ is
essentially complete. We now briefly discuss the case of a general
$\R^\RE$-definable set $A$. Let $I=[-1,1]$, and assume that
$A\subset I^m$ (the general case is easily reduced to this one). Our
approach for this case is based on a quantifier-elimination result of
van den Dries \cite{vdd:Rre} (itself a variant of the work of Denef
and van den Dries \cite{denef-vdd} on subanalytic sets). In
\cite{vdd:Rre} it is shown, up to some minor variations in
formulation, that any $A$ as above is definable by a quantifier-free
formula in a language $L^D_\RE$ which has a natural interpretation in
the structure $I$. This language has an order relation $<$, $m$-ary
operation symbols for certain special functions $f:I^m\to I$, and a
binary operation $D$ called restricted division, interpreted in $I$ as
\begin{equation}
  D(x,y) =
  \begin{cases}
    x/y & |x|\le|y| \text{ and } y\neq0 \\
    0 & \text{otherwise.}
  \end{cases}
\end{equation}
The crucial feature for us is that all functions appearing in the
language extend as holomorphic-Pfaffian functions to a complex
neighborhood of $I^m$. Therefore, a set defined by quantifier-free
$L_\RE$-formulas, i.e. not involving the restricted division $D$, is
essentially the real part of a holomorphic-Pfaffian variety (after
some work to handle inequalities).

To handle a formula involving a restricted division $D(x,y)$, we
replace it by three formulas: one for the case $|x|>|y|$ or $y=0$
where we replace $D(x,y)$ by $0$; one for the case $|x|=|y|\neq0$
where we replace $D(x,y)$ by $1$; and one for this case $|x|<|y|$,
where we replace $D(x,y)$ by a new variable $z$, and add the equation
$zy=x$ (which is equivalent to $z=D(x,y)$ when $x<y$). Repeating this
for every restricted division in the formula, we reduce the set $A$ to
a union of projections (forgetting the variables $z$) of sets $B_j$
definable by quantifier-free $L_\RE$-formulas. Moreover, each fiber of
the projection $\pi:B_j\to A$ contains at most one point: this
corresponds to the fact that we only add a variable $z$ under the
restriction that $|x|<|y|$ and in particular $y\neq0$, and under these
conditions the equation $yz=x$ uniquely defines $z$. We call this type
of projections \emph{admissible} (see
Definition~\ref{def:admissible}).

It remains to study rational points in sets of the form $\pi(B)$,
where $B$ is defined by a quantifier-free $L_\RE$-formula and $\pi$ is
an admissible projection. Up to some minor details involving the
algebraic part of $B$, we may replace $B$ by its complex-analytic germ
-- which is a holomorphic-Pfaffian variety. The strategy described
above for holomorphic-Pfaffian varieties extends in a straightforward
manner to their admissible projections. The relevant statement is
Theorem~\ref{thm:exploring-projection}.

\subsubsection{From rational points to $\cF$-points and points of degree $k$}

The proof can be carried out for $\cF$-points rather than $\Q$-points
in exactly the same manner (we follow the strategy of Pila
\cite[Theorem~3.2]{pila:exp-alg-surface}). The case of algebraic
points of degree $k$ requires some additional work. We essentially
follow the proof of \cite{pila:algebraic-points}, with some minor
additional details needed to obtain the necessary degree estimates.

\subsection{Contents of this paper}

This paper is organized as follows. In~\secref{sec:weierstrass} we
prove some preliminary results of polynomial interpolation in the
complex setting; define the notion of a Weierstrass polydisc; and
establish a result on holomorphic decompositions of functions over a
Weierstrass polydisc. In~\secref{sec:interpolation-determinants} we
give upper and lower bounds for interpolation determinants over a
fixed Weierstrass polydisc, in analogy with the Bombieri-Pila
determinant method. In~\secref{sec:entropy} we recall the notion of
$\e$-entropy and Vitushkin's bound and derive some simple consequences
that are needed in the sequel. In~\secref{sec:pfaffian} we recall the
Pfaffian, semi-Pfaffian and sub-Pfaffian categories; introduce the
holomorphic-Pfaffian category; and prove the key technical result on
covering of holomorphic-Pfaffian varieties by Weierstrass polydiscs.
In~\secref{sec:exploring} we prove an analog of the Wilkie conjecture
for holomorphic-Pfaffian varieties and their projections by an
induction over dimension. In~\secref{sec:definable} we generalize the
results of~\secref{sec:exploring} to arbitrary $\R^\RE$-definable sets
and prove the main Theorems~\ref{thm:main-F} and~\ref{thm:main-k}.
Finally in~\secref{sec:concluding} we give some concluding remarks
related to effectivity and uniformity of the bounds; and discuss
possible generalizations to other structures.

\section{Polynomial interpolation, Weierstrass polydiscs and
  holomorphic decomposition}
\label{sec:weierstrass}

We fix some basic notations. Let $\Omega\subset\C^n$ be a domain and
$Z\subset\Omega$. We denote by $\cO(Z)$ the ring of germs of
holomorphic functions in a neighborhood of $Z$. If $Z$ is relatively
compact in $\Omega$ we denote by $\norm{\cdot}_Z$ the maximum norm on
$\cO(\bar Z)$.

Let $A\subset\C^n$ be a ball or a polydisc around a point $p\in\C^n$
and $\delta>0$. We let $A^\delta$ denote the $\delta^{-1}$-rescaling
of $A$ around $p$, i.e. $A^\delta:=p+\delta^{-1}(A-p)$.

\subsection{Polynomial interpolation}

\subsubsection{Univariate interpolation}

Let $\Omega\subset\C$ be a domain and $f:\Omega\to\C$ a function. Let
$\va:=\{\va_1,\ldots,\va_k\}\subset\Omega$ be a multiset of points and
denote by $\nu(\va_j)$ the number of times $\va_j$ appears in $\va$.
We define the \emph{interpolation polynomial} $L[\va;f]$ to be the
unique polynomial of degree at most $k-1$ satisfying
\begin{equation}
  L[\va;f]^{(l)}(\va_j) = f(\va_j) \qquad j=1,\ldots,k, \quad l=0,\ldots,\nu(\va_j)-1.
\end{equation}
Denote $h_\va(z):=\prod_{j=1}^k(z-\va_j)$. When $f$ is a holomorphic
function, the classical proof of the Weierstrass division theorem (see
e.g. \cite{gr:analytic}) shows that $L[\va;f]$ admits an integral
representation as follows.

\begin{Prop}\label{prop:interpolation-formula}
  Let $\Omega\subset\C^n$ be a simply-connected domain and
  $\va\subset\Omega$. Let $f\in\cO(\bar\Omega)$. Then for $z\in\Omega$,
  \begin{equation}\label{eq:L-integral}
    L[\va;f](z) = \frac1{2\pi i}\oint_{\partial\Omega} \frac{f(\zeta)}{h(\zeta)}
    \frac{h(\zeta)-h(z)}{\zeta-z} \d\zeta, \qquad h:=h_\va.
  \end{equation}
\end{Prop}
\begin{proof}
  The right hand side of~\eqref{eq:L-integral} is easily seen to be a
  polynomial in of degree $k-1$ in $z$ (since this is true for the
  integrand). Evaluating at $z=\va_j$ we have $h(z)=0$, and the integral
  reduces to the Cauchy formula for $f(\va_j)$.
\end{proof}

Next we give norm estimates for $L[\va;f]$ in terms of the norm of
$f$.
\begin{Prop}\label{prop:interpolation-bound}
  Let $D\subset\C$ be a disc, $\va\subset D$ a multiset and and
  $f\in\cO(\bar D^{1/3})$. Then
  \begin{equation}
    \norm{L[\va;f]}_D \le 3\norm{f}_{D^{1/3}}.
  \end{equation}
\end{Prop}
\begin{proof}
  Since the claim is invariant under affine transformations of $\C$ we
  may assume that $D$ is the unit disc. Then we have for any $z\in D$
  an estimate $|h(z)|\le 2^k$ and for any $\zeta\in \partial D^{1/3}$ an
  estimate $|h(\zeta)|\ge(3-1)^k=2^k$. Using the integral
  representation~\eqref{eq:L-integral},
  \begin{multline}
    \norm{L[\va;f]}_D = \max_{z\in D} \abs{\frac1{2\pi
      i}\oint_{\partial D^{1/3}} \frac{f(\zeta)}{\zeta-z}
    \left(1-\frac{h(z)}{h(\zeta)}\right) \d\zeta} \\
    \le 3 \frac{\norm{f}_{D^{1/3}}}{3-1} (1+1) \le 3\norm{f}_{D^{1/3}}.
  \end{multline}
\end{proof}

\subsection{Weierstrass polydiscs}

We say that $\vx=(\vx_1,\ldots,\vx_n)$ is a \emph{standard} coordinate
system on $\C^n$ if it is obtained from the standard coordinates by an
affine unitary transformation.

Let $\Omega\subset\C^n$ be a domain and $X\subset\Omega$ an analytic
subset.
\begin{Def}
  We say that a polydisc $\Delta=\Delta_z\times\Delta_w$ in the
  $\vx=\vz\times\vw$ coordinates is a \emph{pre-Weierstrass polydisc}
  for $X$ if $\bar\Delta\subset\Omega$ and
  $(\bar\Delta_z\times\partial\Delta_w)\cap X=\emptyset$. We call
  $\Delta_z$ the base and $\Delta_w$ the fiber of $\Delta$.

  If $X$ is pure-dimensional, we say that $\Delta$ is a
  \emph{Weierstrass polydisc} for $X$ if $\dim\vz=\dim X$.
\end{Def}
When speaking about (pre-)Weierstrass polydiscs we will assume (unless
otherwise stated) that the coordinates are given by
$\vx=\vz\times\vw$.  We will also denote by
$\pi_z:\C^n\to\C^{\dim\vz}$ the projection to the $\vz$-coordinates and 
by $\pi_z^X$ its restriction to $\Delta\cap X$.

We recall some standard facts.

\begin{Fact}\label{fact:weierstrass-proper}
  If $\Delta$ is a pre-Weierstrass polydisc for $X$ then $\pi_z^X$ is
  proper and finite-to-one.
\end{Fact}
\begin{proof}
  Let $p_i\in\Delta\cap X$ be a sequence of points that escapes to
  infinity and we will show that $\pi_z^X(p_i)$ escapes to infinity in
  $\Delta_z$. Assume otherwise. Passing to a subsequence we may assume
  that $\pi_z^X(p_i)$ converges in $\Delta_z$, and passing to a
  further subsequence we may also assume that $p_i$ converges in
  $\bar\Delta$. But then it must necessarily converge to a point in
  $(\Delta_z\times\partial\Delta_w)\cap X$, which is ruled out by the
  definition of a pre-Weierstrass polydisc. Thus $\pi_z^X$ is proper,
  hence its fibers are compact complex submanifolds of $\Delta_w$, and
  must therefore be finite.
\end{proof}

\begin{Fact}\label{fact:weierstrass-mu-to-1}
  If $X$ has pure dimension $m$ and $\Delta$ is a Weierstrass polydisc
  for $X$ then $\pi_z^X$ is $e(X,\Delta)$-to-1 for some number
  $e(X,\Delta)\in\N$ (where points in the fiber are counted with
  multiplicities).
\end{Fact}
\begin{proof}
  In the Weierstrass case $\dim X=\dim\vz$ so
  $\dim X=\dim \pi^X_z(X)$. Under this condition it is well known that
  the map $\pi_z^X$ is a finite unramified cover outside some proper
  analytic subset $B\subset\Delta_z$ \cite[III.B]{gr:analytic}. Then
  the map is $e(X,\Delta)$-to-1 where $e(X,\Delta)$ is the cardinality
  the fiber over any point in $\Delta_z\setminus B$.
\end{proof}

\begin{Lem}\label{lem:weierstrass-branches}
  Let $X$ have pure dimension $m$ and $\Delta$ be a Weierstrass polydisc
  for $X$. For $l=1,\ldots,\dim\vw$ there exists a monic polynomial
  \begin{equation}
    P_l(z,\vw_l)\in\cO(\Delta_z)[\vw_l], \qquad \deg P_l = e(X,\Delta)
  \end{equation}
  such that for any $z\in\Delta_z$, the roots of $P_l(z,\vw_l)$ are
  precisely the $\vw_l$-coordinates of the points of
  $(\pi_z^X)^{-1}(z)$.
\end{Lem}
\begin{proof}
  In the notations of Fact~\ref{fact:weierstrass-mu-to-1} and its
  proof, set $\nu=e(X,\Delta)$ and let
  $W_1,\ldots,W_\nu:\Delta_z\setminus B\to X\cap\Delta$ be the
  (ramified) inverses of $\pi_z^X$. Then
  \begin{equation}
    P_l(z,\vw_l) = \prod_{j=1}^\nu \big(\vw_l-\vw_l(W_j(z))\big)
  \end{equation}
  has univalued coefficients which are holomorphic outside $B$, and
  since $W_1,\ldots,W_\nu$ are bounded near $B$ it follows from the
  Riemann removable singularity theorem that the coefficients extend
  to holomorphic functions in $\Delta_z$.
\end{proof}

\begin{Prop}\label{prop:weierstrass-interpolation}
  Let $X$ have pure dimension $m$. Let $\Delta$ be a Weierstrass polydisc
  for $X$ and set $\nu=e(X,\Delta)$. Let $f\in \cO(\bar\Delta_z\times\bar\Delta_w^{1/3})$. There exists a
  function
  \begin{equation}
    P\in\cO(\bar\Delta_z)[\vw], \qquad \deg_{\vw_i} P\le\nu-1,\quad i=1,n-m
  \end{equation}
  such that
  $P\rest{X\cap\Delta}=f\rest{X\cap\Delta}$, and
  \begin{equation}\label{eq:weierstrass-interpolation}
    \norm{P}_{\Delta} \le 3^{n-m} \norm{f}_{\Delta_z\times\Delta_w^{1/3}}.
  \end{equation}
\end{Prop}
\begin{proof}
  Set $s:=\dim\vw=n-m$. Let $\Delta_w=\prod_{i=1}^s D_i$. For
  $l=1,\ldots,s$ write
  \begin{gather}
    \hat\vw_l := (\vw_1,\ldots,\vw_{l-1},\w_l,\vw_{l+1},\ldots,\vw_s), \\
    \Omega_l = \Delta_z\times D_1\times\ldots\times D_{l-1}\times D_l^{1/3}\times\ldots\times D_s^{1/3}.
  \end{gather}
  Consider the operator
  \begin{equation}
    L_l (g) = \frac1{2\pi i}\oint_{\partial D_l^{1/3}} \frac{g(z,\hat\vw_l)}{P_l(z,\w_l)}
    \frac{P_l(z,\w_l)-P_l(z,\vw_l)}{\w_l-\vw_l} \d\w_l
  \end{equation}
  where $P_l$ denotes the polynomials from
  Lemma~\ref{lem:weierstrass-branches}.

  We claim that $L_l$ maps $\cO(\bar\Omega_l)$ to
  $\cO(\bar\Omega_{l+1})$. Indeed, let $g\in\cO(\bar\Omega_l)$. For
  every fixed $z\in\Delta_z$ the roots of $P_l(z,\vw_l)$ lie in $D_l$
  and it follows that the integrand is holomorphic whenever
  $(z,w)\in\bar\Omega_{l+1}$ and $\w_l$ lies in a neighborhood of
  $\partial D_l^{1/3}$. By Proposition~\ref{prop:interpolation-bound}
  we have the norm estimate
  \begin{equation}\label{eq:weierstrass-interpolation-1}
    \norm{L_l(g)}_{\bar\Omega_{l+1}} \le 3 \norm{g}_{\bar\Omega_l}.
  \end{equation}

  It is easy to see that if $g$ is polynomial of degree at most $\nu-1$
  in $\vw_j,j\neq l$ then so is $L_l(g)$. Moreover,
  Proposition~\ref{prop:interpolation-formula} shows that $L_l(g)$ is
  polynomial of degree at most $\nu-1$ in $\vw_l$ and agrees with $g$
  for any point $(z,w)\in\bar\Omega_{l+1}$ such that $\vw_l(w)$ is a
  root of $P_l$, and in particular whenever $w\in(\pi_z^X)^{-1}(z)$.

  Finally, setting
  \begin{equation}
    P=L_1\cdots L_s f\in\cO(\bar\Omega_{s+1})=\cO(\bar\Delta)
  \end{equation}
  we obtain a polynomial of degree $\nu-1$ in each variable
  $\vw_1,\ldots,\vw_s$ which agrees with $f$ whenever
  $w\in(\pi_z^X)^{-1}(z)$. The norm estimate~\eqref{eq:weierstrass-interpolation}
  follows by repeated application of~\eqref{eq:weierstrass-interpolation-1}.
\end{proof}

\subsection{Decomposition data}

We recall the following definition from
\cite[Definition~4]{me:interpolation}. Given a standard system of
coordinates $\vx$, we say that $(\Delta,\Delta')$ is a pair of
polydiscs if $\Delta\subset\Delta'$ are two polydiscs with the same
center in the $\vx$ coordinates.

For a co-ideal $\cM\subset\N^n$ and $k\in\N$ we denote by
\begin{equation}
  \cM^{\le k}:=\{ \alpha\in \cM : \abs{\alpha}\le k\}
\end{equation}
and by $H_\cM(k):=\#\cM^{\le k}$ its Hilbert-Samuel function. The function
$H_\cM(k)$ is eventually a polynomial in $k$, and we denote its degree
by $\dim\cM$.

\begin{Def}\label{def:decomposition}
  Let $X\subset\C^n$ be a locally analytic subset, $\vx$ a standard
  coordinate system, $(\Delta,\Delta')$ a pair of polydiscs centered
  at the $\vx$-origin and $\cM\subset\N^n$ a co-ideal. We say that $X$
  admits \emph{decomposition} with respect to the \emph{decomposition
    datum}
  \begin{equation}
     \cD:=(\vx,\Delta,\Delta',\cM)
  \end{equation}
  if there exists a constant denoted $\norm\cD$ such that for every holomorphic
  function $F\in\cO(\bar \Delta')$ there is a decomposition
  \begin{equation}\label{eq:F-decomp}
    F = \sum_{\alpha\in\cM} c_\alpha \vx^\alpha + Q, \qquad Q\in\cO(\bar \Delta)
  \end{equation}
  where $Q$ vanishes identically on $X\cap \Delta$ and
  \begin{equation}\label{eq:F-decomp-norms}
    \norm{c_\alpha \vx^\alpha}_{\Delta} \le  \norm\cD\cdot \norm{F}_{\Delta'} \qquad \forall\alpha\in\cM.
  \end{equation}
  We define the \emph{dimension} of the decomposition datum,
  denoted $\dim\cD$ to be $\dim\cM$.
\end{Def}

Since $\cH_\cM(k)$ is eventually a polynomial of degree $\dim\cM$, the
function $\cH_\cM(k)-\cH_\cM(k-1)$ counting monomials of degree $k$ in
$\cM$ is eventually a polynomial of degree $\dim\cM-1$. If
$\dim\cD\ge1$ we denote by $e(\cD)$ the minimal constant satisfying
\begin{equation}
   H_\cM(k)-H_\cM(k-1) \le e(\cD)\cdot L(\dim\cM,k), \qquad \forall k\in\N.
\end{equation}
where $L(n,k):=\binom{n+k-1}{n-1}$ denotes the dimension of the space
of monomials of degree $k$ in $n$ variables. In the case $\dim\cD=0$
the co-ideal $\cM$ is finite and we denote by $e(\cD)$ its size.

\begin{Lem}\label{lem:taylor-norms}
  Let $\Delta$ be a polydisc in the $\vx$-coordinates and
  $F\in\cO(\bar\Delta)$. Then the Taylor expansion
  \begin{equation}
    F = \sum_\alpha c_\alpha \vx^\alpha
  \end{equation}
  satisfies
  \begin{equation}
    \norm{c_\alpha \vx^\alpha}_\Delta \le \norm{F}_\Delta.
  \end{equation}
\end{Lem}
\begin{proof}
  By rescaling we may assume without loss of generality that
  $\Delta$ is the unit polydisc. Then
  \begin{equation}
    \begin{split}
      \norm{c_\alpha\vx^\alpha}_\Delta &= \abs{c_\alpha} = (2\pi)^{-n}
      \abs{\oiint_{|\vx_1|=\cdots=|\vx_n|=1} F(\vx) \vx^{-\alpha-(1,\ldots,1)} \d\vx_1\wedge\cdots\wedge\vx_n} \\
      &\le \norm{F}_\Delta.
    \end{split}
  \end{equation}
\end{proof}

\begin{Thm}\label{thm:weierstrass-decomp}
  Let $X$ have pure dimension $m$. Let $\Delta$ be a Weierstrass
  polydisc for $X$ and set
  \begin{align}
    \nu&=e(X,\Delta) & \cM&=\N^m\times\{0,\ldots,\nu-1\}^{n-m} & \Delta'&=\Delta_z\times\Delta_w^{1/3}
  \end{align}
  Then $(\vx,\Delta,\Delta',\cM)$ is a decomposition datum for $X$
  with $\norm\cD\le3^{n-m}$, $\dim\cM=m$ and $e(\cD)=\nu^{n-m}$.
\end{Thm}
\begin{proof}
  The statement is a direct corollary of
  Proposition~\ref{prop:weierstrass-interpolation}, taking into account
  Lemma~\ref{lem:taylor-norms}.
\end{proof}

\section{Interpolation determinants}
\label{sec:interpolation-determinants}

Let $\Omega\subset\C^n$ be a domain and $X\subset\Omega$ an analytic
subset of pure dimension $m$. Let $\vx$ be standard coordinates. Let
$\Delta$ be a Weierstrass polydisc for $X$, and set
$\Delta':=\Delta_z\times\Delta_w^{1/3}$ and $\nu:=e(X,\Delta)$ as in
Theorem~\ref{thm:weierstrass-decomp}.

\subsection{Interpolation determinants}

Let $\vf:=(f_1,\ldots,f_\mu)$ be a collection of functions and
$\vp:=(p_1,\ldots,p_\mu)$ a collection of points. We define the
\emph{interpolation determinant}
\begin{equation}
  \Delta(\vf,\vp) := \det (f_i(p_j))_{1\le i,j\le \mu}.
\end{equation}

\begin{Lem}\label{lem:id-upper-bd}
  Assume $m>0$. Suppose $f_i\in\cO(\bar \Delta')$ with
  $\norm{f_i}_{\Delta'}\le M$ and $p_i\in \Delta^{1/\delta}\cap X$ for
  $i=1,\ldots,\mu$ and $0<\delta\le1/2$. Then
  \begin{equation}\label{eq:id-bound}
    \abs{\Delta(\vf,\vp)} \le (C \mu^3 M)^\mu \cdot \delta^{E\cdot\mu^{1+1/m}}
  \end{equation}
  where
  \begin{align}
    C &= O_m(\nu^{\frac{n-m}m}), \\
    E &= \Omega_m(\nu^{-\frac{n-m}m}).
  \end{align}
\end{Lem}
\begin{proof}
  This follows from \cite[Lemma~9]{me:interpolation} and
  Theorem~\ref{thm:weierstrass-decomp}.
\end{proof}

We note that the proof of \cite[Lemma~9]{me:interpolation} is a direct
adaptation of the interpolation determinant method of
\cite{bombieri-pila}, and the reader familiar with this method may
recognize that essentially the same arguments go through given the
definition of decomposition data.

\subsection{Polynomial interpolation determinants}

Let $d\in\N$ and let $\mu$ denote the dimension of the space of
polynomials of degree at most $d$ in $m+1$ variables, $\mu=L(m+2,d)$.
Let $\vf:=(f_1,\ldots,f_{m+1})$ be a collection of functions and
$\vp:=(p_1,\ldots,p_\mu)$ a collection of points. We define the
\emph{polynomial interpolation determinant} of degree $d$ to be
\begin{equation}
  \Delta^d(\vf,\vp) := \Delta(\vg,\vp), \qquad \vg=(\vf^\alpha : \alpha\in\cN^{m+1}, |\alpha|\le d).
\end{equation}
Note that $\Delta^d(\vf,\vp)=0$ if and only if there exists a
polynomial of degree at most $d$ in $m+1$ variables vanishing at the
points $\vf(p_1),\ldots,\vf(p_\mu)$.

Following \cite{pila:exp-alg-surface} we introduce the following
height function. For an algebraic number $\alpha\in\Qa$ we let
$\den(\alpha)$ denotes the \emph{denominator} of $\alpha$, i.e. the
least positive integer $K$ such that $K\alpha$ is an algebraic
integer. If $\{\alpha_i\}$ are the conjugates of $\alpha$ we denote
\begin{equation}
  H^\size(\alpha) = \max(\den(\alpha),|\alpha_i|).
\end{equation}
If $\alpha$ has degree $t$ and
\begin{equation}
   P\in\Z[X], \qquad P=a_t(X-\alpha_1)\cdots(x-\alpha_t)
\end{equation}
is its minimal polynomial then \cite[1.6.5,1.6.6]{bombieri:heights}
\begin{equation}
  H(\alpha)^t = |a_t| \prod_{j=1}^t \max(1,|\alpha_j|).
\end{equation}
In particular it follows that
\begin{equation}\label{eq:H-vs-Hsize}
  H(\alpha)^t \ge H^\size(\alpha).
\end{equation}
For a set $A\subset\C^m$ we define $A^\size(\cF,H)$ in analogy with
$A(\cF,H)$ replacing $H(\cdot)$ by $H^\size(\cdot)$. The following
lemma is essentially contained in the proof of
\cite[Theorem~3.2]{pila:exp-alg-surface}, and we reproduce the
argument for the convenience of the reader.

\begin{Lem}\label{lem:id-lower-bd}
  Let $H\in\N$ and suppose that for every $i=1,\ldots,m+1$ and
  $j=1,\ldots,\mu$,
  \begin{equation}
    f_i(p_j)\in\cF,\qquad H^\size(f_i(p_j)) \le H.
  \end{equation}
  Then $\Delta^d(\vf,\vp)$ either vanishes or satisfies
  \begin{equation}
    \abs{\Delta^d(\vf,\vp)} \ge (\mu! H^{(m+2)\mu d})^{-[\cF:\Q]}.
  \end{equation}
\end{Lem}
\begin{proof}
  Denote the matrix defining $\Delta^d(\vf,\vp)$ by $A$. Let
  $Q_{i,j}:=\den f_i(p_j)$ for $i=1,\ldots,m+1$ and $j=1,\ldots,\mu$.
  By assumption $Q_{i,j}\le H$. The row corresponding to $p_j$ in
  $\Delta^d(\vf,\vp)$ consists of algebraic numbers with common
  denominator dividing $Q_j:=\prod_i Q^d_{i,j}$. Setting
  $K=\prod_{j=1}^\mu Q_j$ we see that $KA$ is a matrix of algebraic
  integers and $|K|\le H^{(m+1)\mu d}$.

  Let $G:=\Gal(\cF/\Q)$. If $\det A$ is non-vanishing then so are its
  $G$-conjugates and then
  \begin{equation}\label{eq:id-lower-1}
    1 \le |\prod_{\sigma\in G} KA^\sigma| = K^{[\cF:\Q]}\cdot |\det A|\cdot \prod_{\id\neq\sigma\in G} |\det(A^\sigma)|.
  \end{equation}
  We estimate $|\det (A^\sigma)|$ from above. By assumption each entry
  of $A^\sigma$ has absolute value bounded by $H^d$. Expanding the determinant
  by the Laplace expansion we have
  \begin{equation}\label{eq:id-lower-2}
    |\det (A^\sigma)| \le \mu! H^{\mu d} \qquad \forall\sigma\in G.
  \end{equation}
  Plugging~\eqref{eq:id-lower-2} into~\eqref{eq:id-lower-1} we have
  \begin{equation}
    \det A \ge K^{-[\cF:\Q]} (\mu! H^{\mu d})^{-[\cF:\Q]+1} \ge (\mu! H^{(m+2)\mu d})^{-[\cF:\Q]}.
  \end{equation}
\end{proof}

Comparing Lemmas~\ref{lem:id-upper-bd} and~\ref{lem:id-lower-bd} we
obtain the following.

\begin{Prop}\label{prop:hypersurface-select}
  Let $M,H\ge2$, and suppose $f_i\in\cO(\bar\Delta')$ with
  $\norm{f_i}_{\Delta'}\le M$. Let
  \begin{equation}
    Y = \vf(X\cap\Delta^2) \subset \C^{m+1}.
  \end{equation}
  There exist a constant $C_n>0$ depending only on $n$ such that if
  \begin{equation}\label{eq:hypersurface-cond}
    d > C_n \nu^{n-m} \left([\cF:\Q]\log H +\log M\right)^m
  \end{equation}
  then $Y^\size(\cF,H)$ is contained in an algebraic hypersurface of degree
  at most $d$ in $\C^{m+1}$.
\end{Prop}
\begin{proof}
  We consider first the case $m=0$. In this case $\nu=e(X,\Delta)$ is
  the number of points in $X\cap\Delta$. In particular this bounds the
  number of points in $Y$, all the more in $Y^\size(\cF,H)$, and the
  claim holds with any $d\ge\nu$.

  Now assume $m>0$ and suppose toward contradiction that
  $Y^\size(\cF,H)$ is not contained in an algebraic hypersurface of
  degree at most $d$ in $\C^{m+1}$. By standard linear algebra it
  follows that there exist $\vp=p_1,\ldots,p_\mu\in X\cap\Delta^2$
  such that $\{\vf(p_j):j=1,\ldots,\mu\}$ is a subset of $Y$ and does
  not lie on the zero locus of any non-zero polynomial of degree $d$.
  Then $\abs{\Delta^d(\vf,\vp)}\neq0$, and from
  Lemmas~\ref{lem:id-upper-bd} and~\ref{lem:id-lower-bd} we have
  \begin{equation}
    (\mu! H^{(m+2)\mu d})^{-[\cF:\Q]} \le \abs{\Delta^d(\vf,\vp)} \le (C \mu^3 M^d)^\mu \cdot (1/2)^{E\cdot\mu^{1+1/m}}.
  \end{equation}
  Taking logs and using $\mu\sim_m d^{m+1}$ we have
  \begin{equation}
    \log 2\cdot E\cdot d^{1+1/m} \lesssim \log(C\mu^3 M^d)+[\cF:\Q]\big[(m+2)d \log H+\log\mu \big].
  \end{equation}
  Therefore
  \begin{equation}
    d^{1/m} = \nu^{\frac{n-m}m}
    O_n\left( \frac{\log \nu}{d}+\log M+[\cF:\Q]\log H \right).
  \end{equation}
  Finally note that~\eqref{eq:hypersurface-cond} implies, in the case
  $m>0$, that $(\log\nu)/d=O_n(1)$.
\end{proof}

\section{Metric entropy, Vitushkin's bound}
\label{sec:entropy}

Let $A\subset\R^n$ be a relatively compact subset. For every $\e>0$ we
denote by $M(\e,A)$ the minimal number of closed balls of radius $\e$
needed to cover $A$. The logarithm of $M(\e,A)$ is called the
\emph{$\e$-entropy} of $A$.

For $r>0$ we denote $Q_r:=[0,r]\subset\R$. In our setting it will be
more convenient to define $M(\e,A)$ in terms of covering by
$\e$-cubes, i.e. translates of the cube $Q_\e^n$. For simplicity we
will also restrict our considerations to the unit cube
$Q_1^n\subset\R^n$.

Vitushkin's bound states that
\begin{equation}
  M(\e,A) \le c_n \sum_{i=0}^n \tilde V_i(A)/\e^i,
\end{equation}
where $\tilde V_i(A)$ denote the $i$-th variation of $A$, that is the
average number of connected components of the section $A\cap P$ over
all affine $(n-i)$-planes $P\subset\R^n$ with respect to an
appropriate measure.

Let $A\subset Q_1^n$ and denote by $V_i(A)$ the maximal number of
connected components of the set $A\cap P$ where $P\subset\R^n$ is an
affine $(n-i)$-plane (or $\infty$ if this number is unbounded). We
also denote $V(A):=\max_i V_i(A)$. We will use the following result of
Friedland and Yomdin \cite{yomdin-friedland}.

\begin{Thm}[\protect{\cite[Theorem~1]{yomdin-friedland}}]\label{thm:vitushkin}
  Let $A\subset Q_1^n$ and $0<\e\le1$. Then
  \begin{equation}
    M(\e,A) \le \vol(A) + \sum_{i=0}^n 2^i \binom{n}{i} V_i(\partial A)/\e^i.
  \end{equation}
\end{Thm}

We use the following to slightly improve the asymptotics, but
it is otherwise inessential.

\begin{Cor}\label{cor:vitushkin}
  Let $A\subset Q_1^n$ be subanalytic and suppose $\dim A\le m<n$.
  Then
  \begin{equation}
    M(\e,A) \le \sum_{i=0}^m 2^i \binom{n}{i} V_i(A)/\e^i.
  \end{equation}
\end{Cor}
\begin{proof}
  Note first that in this case $\partial A=A$. In the proof of
  Theorem~\ref{thm:vitushkin} for every fixed $\e$ the quantity
  $V_i(A)$ is in fact only used to estimate the number of connected
  components of the intersection $A\cap P$ where $P$ varies over a
  certain finite set of affine $(n-i)$-planes $P$. It is easy to see
  that the argument remains valid if one replaces each $P$ by its
  sufficiently small parallel translate $P'$. For $i>m$ we can choose
  these translates so that $A\cap P'=\emptyset$, and the statement
  follows.
\end{proof}

\begin{Cor}\label{cor:ball-selection}
  Let $r>0$ and $A\subset Q_r^n$ with $\dim A=m<n$. If
  \begin{equation}
    r\e^{-1} > \sqrt[n-m]{CV(A)}, \qquad C:=(m+1)2^{4n}.
  \end{equation}
  then there exists an $\e$-ball disjoint from
  $A$.
\end{Cor}
\begin{proof}
  Since the claim is invariant under rescaling we may assume $r=1$.
  Let $S\subset Q_1^n$ be a set of at least $(4\e)^{-n}$ points with
  pairwise $\ell_\infty$ distances at least $4\e$: for instance one
  can choose a grid with $(4\e)^{-1}$ equally spaced points on each
  axis. Suppose $A$ touches the $\e$-ball $B_s$ around each point of
  $s\in S$. Then every $\e$-cover of $A$ by cubes must contain a cube
  that touches each $B_s$, and since an $\e$-cube cannot touch two
  such balls by the triangle inequality it follows that
  \begin{equation}
    (4\e)^{-n} \le M(\e,A) \le \sum_{i=0}^m 2^i \binom{n}{i} V_i(A)/\e^i
    \le 2^{2n}(m+1) \e^{-m} V
  \end{equation}
  and the conclusion follows.
\end{proof}

\section{The Pfaffian category, Entropy and Weierstrass polydiscs}
\label{sec:pfaffian}

\subsection{Pfaffian functions, semi-Pfaffian and sub-Pfaffian sets}

Let $U\subset\R^n$ be a domain. We denote the coordinates on $\R^n$ by
$\vx$. The following definition, which plays a key role in our
considerations, was introduced by Khovanskii in
\cite{khovanskii:fewnomials} (see also \cite{gv:complexity}).

\begin{Def}
  A \emph{Pfaffian chain} of order $\ell$ and degree $\alpha$ is a sequence
  of function $f_1,\ldots,f_\ell:U\to\R$, real analytic in $U$ and satisfying
  a triangular system of differential equations
  \begin{equation}
    \d f_j = \sum_{i=1}^n P_{i,j}(\vx,f_1(\vx),\ldots,f_j(\vx)) \d x_i, \qquad j=1,\ldots,\ell
  \end{equation}
  where $P_{ij}$ are polynomials of degrees not exceeding $\alpha$. A function
  $f:U\to\R$ of the form $f(\vx)=P(\vx,f_1(\vx),\ldots,f_\ell(\vx))$ where $P$ is a polynomial
  of degree not exceeding $\beta$ is called a \emph{Pfaffian function}
  of \emph{order} $\ell$ and \emph{degree} $(\alpha,\beta)$.
\end{Def}


The following Pfaffian analog of the Bezout theorem, due to Khovanskii
\cite{khovanskii:fewnomials}, is the basis for the theory of Pfaffian
functions and sets.

\begin{Thm}\label{thm:pfaff-points}
  Let $f_1,\ldots,f_n:U\to\R$ be Pfaffian functions with a common
  Pfaffian chain of order $\ell$ and $\deg f_i=(\alpha,\beta_i)$. Then the number
  of isolated points in $\{\vx\in U:f_1(\vx)=\cdots= f_n(\vx)=0\}$ does not exceed
  \begin{equation}
    2^{\ell(\ell-1)/2} \beta_1\cdots\beta_n
    (\min(n,\ell)\alpha+\beta_1+\cdots+\beta_n-n+1)^\ell.
  \end{equation}
\end{Thm}

We now move to the notion of semi-Pfaffian and sub-Pfaffian sets. For
this purpose we restrict our consideration to domains of the form
$\prod_{j=1}^n I_j$ where each $I_j$ is an open, possibly unbounded
interval in $\R$. By a slight abuse of notation we denote this product
by $\cI^n$. We will write $X^n:=(X_1,\ldots,X_n)$ for a set a variables
ranging over $\cI^n$.

\begin{Def}\leavevmode 
  \begin{itemize}
  \item A \emph{basic Pfaffian relation} on $\cI^n$ is a relation
    $f(X^n)*0$ where $*\in\{=,>\}$ and $f$ is a Pfaffian function on
    $\cI^n$.
  \item A \emph{semi-Pfaffian formula} $\phi(X^n)$ is a Boolean
    combination of basic Pfaffian relations. We say that $\phi$ has
    complexity $(n,s,\ell,\alpha,\beta)$ if it involves $s$ basic
    Pfaffian relations, where all the Pfaffian functions have degree
    at most $\beta$ in a common Pfaffian chain of order $\ell$ and
    degree $\alpha$.
  \item A \emph{sub-Pfaffian formula} is a formula of the form
    $\phi(X^n):=\exists Y^r:\psi(X^n,Y^r)$ where $\psi(X^n,Y^r)$ is a
    semi-Pfaffian formula on $\cI^{n+r}$. The complexity of $\phi$ is
    defined to be $(n,r,s,\ell,\alpha,\beta)$ where $\psi$ has
    complexity $(n+r,s,\ell,\alpha,\beta)$.
  \end{itemize}
  If a formula is semialgebraic then we omit $\ell,\alpha$ from the
  complexity notation (formally $\ell=\alpha=0$).
\end{Def}

We write $\phi(\cI^n)$ for the set of points in $\cI^n$ satisfying
$\phi$, and refer to such sets as semi-Pfaffian (resp. sub-Pfaffian)
for $\phi$ semi-Pfaffian (resp. sub-Pfaffian). The categories of
semi-Pfaffian and sub-Pfaffian sets thus defined admit effective
estimates for various geometric quantities in terms of the complexity
of the formulas. We will require only estimates for the number of
connected components, which are provided by the following theorem.

\begin{Thm}[\protect{\cite[Theorem~6.6]{gv:compact-approx}}]\label{thm:pfaffian-complexity}
  If $\phi$ is semi-Pfaffian of complexity $(n,s,\ell,\alpha,\beta)$
  then the number of connected components of $\phi(\cI^n)$ is bounded
  by
  \begin{equation}
    s^n 2^{\ell(\ell-1)/2} O(n\beta+\min(n,\ell)\alpha)^{n+\ell}.
  \end{equation}
  Similarly if $\phi$ is sub-Pfaffian of complexity
  $(n,r,s,\ell,\alpha,\beta)$ then the number of connected components
  of $\phi(\cI^n)$ is bounded by
  \begin{equation}
    s^{n+r} 2^{\ell(\ell-1)/2} O((n+r)\beta+\min(n,\ell)\alpha)^{n+r+\ell}.
  \end{equation}
\end{Thm}
\begin{proof}
  The first part is \cite[Theorem~6.6]{gv:compact-approx}. The second
  follows from the first since projection cannot increase the number
  of connected components.
\end{proof}

\begin{Cor}\label{cor:pfaffian-V}
  Let $\phi$ be sub-Pfaffian of complexity
  $(n,r,s,\ell,\alpha,\beta)$. Then $V(\phi(\cI^n))$ is bounded by a
  polynomial of degree at most $n+r+\ell$ in $\beta$.
\end{Cor}
\begin{proof}
  To estimate $V(\phi(\cI^n))$ we intersect with additional
  linear equations and count connected components. The result
  follows easily from Theorem~\ref{thm:pfaffian-complexity}.
\end{proof}

\subsection{Sub-Pfaffian sets and $\R^\RE$}

The restricted exponential and sine functions are Pfaffian. As a
consequence we have the following proposition.

\begin{Prop}\label{prop:Rre-subPfaffian}
  Every $\R^\RE$-definable subset of $\R^n$ is sub-Pfaffian.
\end{Prop}
\begin{proof}
  By the main result of \cite{vdd:Rre} every $\R^\RE$-definable subset
  of $\R^n$ is definable by a formula of the type
  \begin{equation}
    \phi(X^n) := \exists Y^m : \psi(X^n,Y^m)
  \end{equation}
  where $\psi$ is a quantifier-free $\R^\RE$-formula (in fact one can
  replace $\exists$ by ``exists a unique'', although we shall not use
  this fact). By adding additional variables $Y$ one can also assume
  that the function symbols $\exp,\sin$ only appear in the form
  $\exp(Y_j),\sin(Y_j)$: by induction on the construction tree of each
  term we replace every occurrence of $\exp(T)$ for a term $T$ by
  $\exp(Y_j)$ for some new variable $Y_j$, and add the condition
  $Y_j=T$ to $\psi$ (and similarly for $\sin$). Then it will follow
  that $\psi$ is equivalent to a sub-Pfaffian formula once we show
  that the graph of the restricted exponential and sine functions is
  sub-Pfaffian.

  It is known that the function $\sin(z)$ is Pfaffian in the interval
  $[0,\pi]$. We claim that the graph of the restricted sine,
  $X_2=\sin\rest{[0,\pi]}(X_1)$ in $\cI^2:=\R^2$ is sub-Pfaffian. Note
  that in defining this graph we may not use the function $\sin(X_1)$,
  since $\sin$ is not a Pfaffian function in $\R$. To resolve this
  minor technicality we define the graph by a projection from
  $\cI^3:=\R^2\times[0,\pi]$ using the sub-Pfaffian formula
  \begin{multline}
    \phi_{\sin}(X_1,X_2) := [(X_1<0\lor X_1>\pi)\land X_2=0]\lor\\
    [\exists Y\in[0,\pi]:(X_1=Y\land X_2=\sin(Y))]
  \end{multline}
  where $\sin Y$ is a Pfaffian function over $[0,\pi]$. The restricted
  exponential function can be treated similarly (in fact here it is
  not necessary to add the additional variable over $[0,1]$ because
  $\exp$ is Pfaffian in $\R$ itself).
\end{proof}

\subsection{Pfaffian functions in the complex domain}

We return now to the complex setting. We fix some standard
coordinates $\vx$ on $\C^n$ and identify $\C^n$ with
$\R^{2n}$ by the map
\begin{equation}
  (\vx_1,\ldots,\vx_n)\to(\Re\vx_1,\Im\vx_1,\ldots,\Re\vx_n,\Im\vx_n).
\end{equation}
Since we work in $\C^n$ it will be convenient to allow unitary changes
of variables. To make this consistent with the Pfaffian framework we
consider the following setting. We let $U$ denote some fixed ball
around the origin, and we will assume that our Pfaffian chain is
defined over $U$. We then let $\cI^{2n}$ denote some product of
intervals and assume $A\cdot\cI^{2n}\subset U$ for any unitary $A$.
Finally we will always work with sub-Pfaffian sets contained in a ball
$B\subset\cI^n$ and we assume that the formulas explicitly contain the
condition $\vx\in B$. Under these assumptions we can make a constant
unitary change of variable in a Pfaffian formula without affecting the
complexity: if $\{f_j(\vx)\}$ is a Pfaffian chain then by the chain
rule $\{f_j(A\cdot\vx)\}$ is a Pfaffian chain of the same order $\ell$
and degree $\alpha$. If the coefficients of $A$ are taken to be
independent variables then this transformation increases the degree
$\alpha$ by $1$.

Our main result in this subsection is a theorem showing that if an
analytic set $X$ in a ball $B\subset\C^n$ is sub-Pfaffian, then one
can choose a Weierstrass polydisc for $X$ with size depending
polynomially on $\beta$. Since we are mainly concerned with the
asymptotic in $\beta$, we allow the asymptotic constants to depend on
all other parameters. In particular when dealing with formulas of
complexity $(2n,r,s,\ell,\alpha,\beta)$ we view all parameters except
$\beta$ as $O(1)$.

We will require a slight technical extension of the notion of
Weierstrass polydiscs.

\begin{Def}
  Suppose $\Delta:=\Delta_z\times\Delta_w$ is a Weierstrass polydisc
  for an analytic set $X$. If $B\subset\C^n$ is a Euclidean ball
  around the origin, we will say that $\Delta$ has \emph{gap} $B$ if
  $\Delta_z\times\partial\Delta_w$ is disjoint from the set $B+X$.
\end{Def}

We begin with a lemma in codimension one.

\begin{Lem}\label{lem:pfaff-polydisc-codim1}
  Let $B\subset\C^n$ be a Euclidean ball around the origin. Let
  $X\subset B$ be an analytic subset of pure dimension $m$. Suppose
  $X$ is defined by a sub-Pfaffian formula $\phi$ of complexity
  $(2n,r,s,\ell,\alpha,\beta)$. Then there exists a pre-Weierstrass
  polydisc $\Delta:=\Delta_z\times\Delta_w$ centered at the origin for
  $X$ where $\dim\vw=1$ and $B^\eta\subset\Delta\subset B$ where
  \begin{align}
      \eta &:= O(\beta^\nu) &
      \nu = \nu(n,m,r,\ell) &:= \frac{2n+r+\ell+2}{2n-2m-1}.
  \end{align}
  Moreover $\Delta$ can be chosen to have gap $B^\eta$.
\end{Lem}
\begin{proof}
  We may assume without loss of generality that $B$ is the unit ball.
  The group $S_1:=\{\zeta\in\C:|\zeta|=1\}$ acts on $B$ by
  multiplication. We consider $Z:=S^1\cdot X$, the $S^1$-saturation of
  $X$. Then $Z$ can be defined as a sub-Pfaffian set using the formula
  \begin{equation}
    \psi(\vx) := \exists(\zeta\in\C) : (|\zeta|^2=1)\land\phi(\zeta\cdot y).
  \end{equation}
  of complexity $(2n,r+2,O(1),\ell,O(1),\beta)$. By
  Corollary~\ref{cor:pfaffian-V} we have
  \begin{equation}
    V(Z)=O(\beta^{2n+r+\ell+2}).
  \end{equation}
  Note that $Z$ has real dimension at most $2m+1$. Then according to
  Corollary~\ref{cor:ball-selection} applied to $Z$ in the cube
  $Q:=[1/4n,1/2n]^{2n}\subset B$ there exists a ball $B_v\subset Q$
  with center $v\in Q$ and radius $\Omega(\beta^{-\nu})$ such that
  $B_v\cap Z=\emptyset$. Equivalently,
  $(S^1\cdot B_v)\cap X=\emptyset$.
  
  Making a unitary change of coordinates we may assume that in the
  $\vx=\vz\times\vw$ coordinates $v$ is given by $(0,\l)$ where
  $|\l|=\Omega(1)$. Let $\Delta_w$ denote the disc of radius $|\l|$
  around the origin in the $\vw$ coordinate and $\Delta_z$ denote a
  polydisc of polyradius $\Omega(\beta^{-\nu})$ around the origin in
  the $\vz$ coordinates with $v+\Delta_z\subset B_v$. Since $\Delta_z$
  is invariant under the $S^1$ action,
  \begin{equation}
    \Delta_z\times\partial\Delta_w = (S^1\cdot v)+\Delta_z = S^1\cdot(v+\Delta_z) \subset  S^1\cdot B_v
  \end{equation}
  is disjoint from $X$, i.e. $\Delta:=\Delta_z\times\Delta_w$ is a
  Weierstrass polydisc for $X$. Finally, since each radius of $\Delta$
  is $\Omega(\beta^{-\nu})$ we have $B^\eta\subset\Delta$ for
  $\eta=O(\beta^{-\nu})$ as claimed.

  To satisfy the gap condition it is enough to choose $B_v$ to be
  disjoint from $Z+B^{O(\beta^\nu)}$ instead of $Z$. This is clearly
  possible for the same reasons: for instance it is enough to decrease
  the radius of $B_v$ by a factor of two.
\end{proof}

We now state our main result.

\begin{Thm}\label{thm:pfaff-polydisc}
  Let $B\subset\C^n$ be a Euclidean ball around the origin. Let
  $X\subset B$ be an analytic subset of pure dimension $m$. Suppose
  $X$ is defined by a sub-Pfaffian formula $\phi$ of complexity
  $(2n,r,s,\ell,\alpha,\beta)$. Then there exists a Weierstrass
  polydisc $\Delta:=\Delta_z\times\Delta_w$ centered at the origin for
  $X$ where $B^\eta\subset\Delta\subset B$ and
  \begin{align}
      \eta &:= O(\beta^\theta) &
      \theta = \theta(n,m,r,\ell) &\le (2n+r+\ell+2)\log (2n-2m+1).
  \end{align}
  Moreover $\Delta$ can be chosen to have gap $B^\eta$.
\end{Thm}
\begin{proof}
  We begin with a simple topological remark. Suppose that
  $\Delta_1=\Delta_z\times\Delta_w\subset B$ is a pre-Weierstrass polydisc for
  $X$. Then $\pi^X_z$ is proper so $X':=\pi^X_z(X)\subset\Delta_z$ is
  an analytic subset. Suppose $\Delta_2:=\Delta_{z'}\times\Delta_{w'}\subset\Delta_z$
  is a Weierstrass polydisc for $X'$. Then
  \begin{equation}
    \Delta:=\Delta_2\times\Delta_w\subset\Delta_1
  \end{equation}
  is a Weierstrass polydisc for $X$. Indeed,
  \begin{itemize}
  \item $X$ does not meet
    $\Delta_{z'}\times(\Delta_{w'}\times\partial\Delta_w)$ since it
    does not meet $\Delta_z\times\partial\Delta_w$.
  \item $X$ does not meet
    $\Delta_{z'}\times(\partial\Delta_{w'}\times\Delta_w)$ since its
    $z$-projection $X'$ does not meet
    $\Delta_{z'}\times\partial\Delta_{w'}$.
  \end{itemize}

  We now proceed with the proof, by induction on $n-m$. The case
  $n-m=1$ is exactly Lemma~\ref{lem:pfaff-polydisc-codim1}. For
  $n-m>1$, consider the pre-Weierstrass polydisc
  $\Delta_1=\Delta_z\times\Delta_w\subset B$ provided by
  Lemma~\ref{lem:pfaff-polydisc-codim1}. Choose some ball
  $B'\subset\Delta_z$. Then $X'\cap B'\subset B'$ is a sub-Pfaffian
  set: after a unitary change from the $\vx$ to the $\vz\times\vw$
  coordinates it is defined by the formula
  \begin{equation}
    \psi(\vz) = \exists\vw : (\vw\in\Delta_w)\land(\vz\in B')\land\phi(\vz,\vw).
  \end{equation}
  of complexity $(2n-2,r+2,O(1),\ell,O(1),\beta)$. Thus, applying the
  inductive hypothesis we obtain a Weierstrass polydisc
  $\Delta_2\subset B'\subset\Delta_z$ for $X'$. Defining $\Delta$ as
  above we obtain a Weierstrass polydisc for $X$.

  We have $B^{\eta_1}\subset \Delta_1$ where $\eta_1=O(\beta^\nu)$ for
  $\nu=\nu(n,m,r,\ell)$. Then $B'$ can be chosen so that
  $B^{O(\eta_1)}\subset B'\times\Delta_w$. Also
  $(B')^{\eta_2}\subset \Delta_2$ where $\eta_2=O(\beta^\theta)$ for
  $\theta=\theta(n-1,m,r+2,\ell)$. Setting
  \begin{equation}
    \eta = O(\eta_1)\cdot \eta_2 = O(\beta^{\nu+\theta})
  \end{equation}
  we see that
  \begin{equation}
    B^\eta = (B^{O(\eta_1)})^{\eta_2} \subset (B'\times\Delta_w)^{\eta_2} \subset
    (B')^{\eta_2}\times\Delta_w\subset\Delta_2\times\Delta_w=\Delta.
  \end{equation}
  Finally computing $\theta$ by induction we see
  \begin{equation}
    \begin{split}
      \theta(n,m,r,\ell) &= \sum_{j=0}^{n-m-1} \nu(n-j,m,r+2j,\ell) =
      \sum_{j=0}^{n-m-1} \frac{2n+r+\ell+2}{2n-2j-2m-1} \\
      &\le (2n+r+\ell+2)\log (2n-2m+1).
    \end{split}
  \end{equation}

  To verify the gap condition, by
  Lemma~\ref{lem:pfaff-polydisc-codim1} we may choose $\Delta_1$ to
  have gap $B^{\eta_1}$ and by induction we may choose $\Delta_2$ to
  have gap $(B')^{\eta_2}$. Then
  \begin{equation}
    [\Delta_{z'}\times(\Delta_{w'}\times\partial\Delta_w)]\cap(X+B^{\eta_1})
    \subset (\Delta_z\times\partial\Delta_w)\cap(X+B^{\eta_1}) = \emptyset
  \end{equation}
  and
  \begin{multline}
    [\Delta_{z'}\times(\partial\Delta_{w'}\times\Delta_w)]\cap(X+(B'\times\Delta_w)^{\eta_2}) \\
    \subset \pi_z^{-1} [(\Delta_{z'}\times\partial\Delta_{w'}) \cap (X'+(B')^{\eta_2})] = \emptyset.
  \end{multline}
  Since $B^\eta\subset B^{\eta_1}$ and
  $B^\eta\subset (B'\times\Delta_w)^{\eta_2}$ we see that $\Delta$
  indeed has gap $B^\eta$.
\end{proof}

Theorem~\ref{thm:pfaff-polydisc} contains the key argument that allows
us to cover analytic sets by a polynomial (in $\beta$) number of Weierstrass
polydiscs. However, in practice the condition of pure-dimensionality of
$X$ is somewhat inconvenient. We use a deformation argument to obtain
a result valid in the case of mixed dimensions. We begin with a definition.

\begin{Def}\label{def:hol-pfaff-variety}
  Let $\Omega\subset\C^n$ be a domain and
  $\{g_l:\Omega\to\C, l=1,\ldots,S\}$ a collection of holomorphic
  functions. Suppose that the graphs of $g_\alpha$ are sub-Pfaffian
  with complexity bounded by $(2n+2,r,s,\ell,\alpha,\beta)$. Then we
  say that the analytic set $X\subset\Omega$ of common zeros of
  $\{g_l\}$ is a \emph{holomorphic-Pfaffian variety}.
\end{Def}

\begin{Cor}\label{cor:pfaff-polydisc}
  Let $X\subset\Omega$ be a holomorphic-Pfaffian variety as in
  Definition~\ref{def:hol-pfaff-variety} and $B\subset\Omega$ a
  relatively compact Euclidean ball. Let $0\le m<n$. There exists an
  analytic set $Z\subset B$ of pure dimension $m$ satisfying
  $X^{\le m}\subset Z$ and a Weierstrass polydisc $\Delta$ for $Z$
  such that:
  \begin{enumerate}
  \item $B^\eta\subset\Delta\subset B$ where $\eta=O(\beta^\theta)$
    and
    \begin{equation}
      \theta = \theta(n,m,2S(r+1),\ell) \le (2n+2S(r+1)+\ell+2)\log (2n-2m+1).
    \end{equation}
  \item $e(Z,\Delta)=O(\beta^{n+2S(r+1)+\ell})$.
  \end{enumerate}
\end{Cor}
\begin{proof}
  The claim is invariant under translation and we may assume without
  loss of generality that $B$ is centered at the origin. Let
  $\tilde g_1$ denote a generic linear combination of the $g_l$.
  If not all $g_l$ are identically vanishing then the zero locus
  $Z_1$ of $\tilde g_1$ is an analytic subset of $\bar B$ of
  codimension $1$. In particular it has finitely many irreducible
  components. We write $Z_{1,b}$ for the union of those components of
  $Z_1$ which are components of $X$ and $Z_{1,g}$ for the rest.

  For every component of $Z_{1,g}$ there is a function $g_l$
  which is not identically vanishing on it. Then we may choose a
  generic linear combination $\tilde g_2$ of the $g_l$ which is
  not identically vanishing on any component of $Z_{1,g}$. We set
  $Z_2=Z_{1,g}\cap\{\tilde g_2=0\}$, which is an analytic subset of
  $\bar B$ of codimension $2$. We write $Z_{2,b}$ for the union of
  those components of $Z_2$ which are components of $X$ and $Z_{2,g}$
  for the rest.

  Proceeding in the same manner we obtain a set $Z:=Z_{n-m}$
  with
  \begin{equation}
    X\subset Z\cup Z_b, \qquad Z_b:= Z_{1,b}\cup\cdots\cup Z_{n-m-1,b}.
  \end{equation}
  Since the components of the sets $Z_{j,b}$ for $j=1,\ldots,n-m-1$
  have codimension $j<n-m$ we have $X^{\le m}\subset Z$.

  Fix a tuple $c_1,\ldots,c_{n-m}\in\C$. Let $\e>0$ and set
  \begin{equation}
    Z^\e := \{\vx\in B: \tilde g_1(\vx)=c_1\e, \cdots, \tilde g_{n-m}(\vx)=c_{n-m}\e\}.
  \end{equation}
  By a Sard-type argument, for generic $c_j$ and sufficiently small
  $\e>0$ we have $\dim Z_\e=m$. We claim that $Z$ is contained in the
  Hausdorff limit of $Z_{\e_j}$ along any sequence $0\neq\e_j\to0$.
  Note that $Z$ has pure dimension $m$ while $Z_b\cap Z$ has dimension
  strictly smaller than $m$, so $Z_g:=Z\setminus Z_b$ is dense in $Z$
  and it will suffice to prove that $Z_g$ is contained in the limit of
  $Z_\e$. Let $y\in Z_g$ and we will show it is in the limit of
  $Z_{\e_j}$.

  The equations $\tilde g_1=\cdots=\tilde g_{n-m}=0$ intersect
  properly, i.e. at an analytic set of dimension $m$, around $y$. If
  we choose $m$ additional generic affine-linear functions
  $L_1,\ldots,L_m$ vanishing at $y$ then the intersection
  \begin{equation}
    \tilde g_1=\cdots=\tilde g_{n-m} = L_1=\cdots=L_m=0
  \end{equation}
  is a proper isolated intersection. By conservation of proper
  intersection numbers under deformations we see that the system
  \begin{equation}
    \tilde g_1=c_1\e_j ,\ldots, \tilde g_{n-m}=c_{n-m}\e_j, \quad L_1=\cdots=L_m=0
  \end{equation}
  must indeed admit at least one solution $y_j\in Z^{\e_j}$ converging
  to $y$ as $\e_j\to0$.

  For $\e>0$ the set $Z^\e$ is defined by the sub-Pfaffian formula
  \begin{equation}\label{eq:cor-pfaff-polydisc-1}
    \phi(\vx) = (x\in B)\land \exists(y_1,\ldots,y_S): \bigwedge_{l=1}^S (y_l=g_l(\vx))
    \cap\bigwedge_{j=1}^{n-m} (\tilde g_j(\vx)=c_j\e)
  \end{equation}
  where we write each $\tilde g_j$ as an appropriate linear
  combination of the $y_l$ variables. The complexity of $\phi$ is
  bounded by $(2n,2S(r+1),O(1),\ell,O(1),\beta)$. By
  Theorem~\ref{thm:pfaff-polydisc}, $Z^\e$ admits a Weierstrass
  polydisc $\Delta(\e)$ satisfying $B^\eta\subset\Delta(\e)\subset B$ where
  \begin{align}
      \eta &:= O(\beta^\theta) &
      \theta = \theta(n,m,2S(r+1),\ell).
  \end{align}
  Moreover we may assume that the $\Delta(\e)$ have a gap bounded
  from below uniformly over $\e$.

  Since the space of Weierstrass polydiscs satisfying the conditions
  above is compact we may choose a sequence $\e_j\to0$ such that
  $\Delta(\e_j)$ converges (for instance in the Hausdorff distance) to
  some polydisc $\Delta$. We claim the $\Delta$ is a Weierstrass
  polydisc for $Z$. Indeed, suppose $Z$ intersects
  $\Delta_z\times\partial\Delta_w$. Since $Z$ is the Hausdorff limit
  of $Z^{\e_j}$ we see that points of $Z^{\e_j}$ must come arbitrarily
  close to $\Delta_z\times\partial\Delta_w$. But this contradicts the
  fact that $\Delta(\e_j)$ converges to $\Delta$ and $Z^{\e_j}$ stays
  at a uniformly bounded distance from
  $\Delta(\e_j)_z\times\partial\Delta(\e_j)_w$.

  To estimate $e(Z,\Delta)$ recall that $Z\setminus Z_b$ has dimension
  strictly smaller than $m$. Since the map $\pi^Z_z$ is finite we see
  that for a generic choice of a point $p\in\Delta_z$ the fiber
  $(\pi^Z_z)^{-1}(p)$ consists of $\nu=e(Z,\Delta)$ isolated points in
  $Z\setminus Z_b$. Each such isolated point is an isolated solution
  of the system
  \begin{equation}\label{eq:cor-pfaff-polydisc-2}
    \{ (\vz,\vw)\in\Delta : \tilde g_1(\vz,\vw)=\cdots=\tilde g_{n-m}(\vz,\vw)=0, \vz=\vz(p) \}.
  \end{equation}
  Then the intersection~\eqref{eq:cor-pfaff-polydisc-2} is proper at these
  isolated points and it follows that for sufficiently small $\e$ the
  intersection
  \begin{equation}
    \{ (\vz,\vw)\in\Delta : \tilde g_1(\vz,\vw)=c_1\e,\ldots=\tilde g_{n-m}(\vz,\vw)=c_{n-m}\e, \vz=\vz(p) \}
  \end{equation}
  contains at least $\nu$ points. But this intersection is
  sub-Pfaffian, being the intersection of $Z^\e$ with the equation
  $\vz=\vz(p)$. Hence the upper bound for $\nu$ follows from
  Theorem~\ref{thm:pfaffian-complexity}.
\end{proof}

\begin{Rem}\label{rem:pfaff-polydisc-S}
  In Corollary~\ref{cor:pfaff-polydisc}, if some of the functions
  $g_l$ are in fact \emph{Pfaffian} of degree $\beta$ rather
  sub-Pfaffian, then one can take $S$ to be the number of sub-Pfaffian
  functions. Indeed, in the formula~\eqref{eq:cor-pfaff-polydisc-1}
  one does not need to add new variables $y_l$ to express the value of
  the Pfaffian $g_l$: as Pfaffian functions they can be summed into
  the linear combinations $\tilde g_j$ directly.
\end{Rem}

\section{Exploring rational points}
\label{sec:exploring}

We begin with a definition.

\begin{Def}\label{def:xw}
  Let $X\subset\C^m$ and $W\subset\C^m$ be two sets. We define
  \begin{equation}
    X(W) := \{ w\in W : W_w \subset X \}
  \end{equation}
  to be the set of points of $W$ such that $X$ contains the germ of
  $W$ around $w$, i.e. such that $w$ has a neighborhood
  $U_w\subset\C^m$ such that $W\cap U_w\subset X$.
\end{Def}

If $A\subset\C^n$ we denote by $A_\R:=A\cap\R^n$. We
remark that
\begin{equation}\label{eq:alg-part-R}
  (A(W))_\R\subset(A_\R)(W_\R).
\end{equation}
We will consider Definition~\ref{def:xw} in two cases: for
$X\subset\C^m$ locally analytic and $W\subset\C^m$ an algebraic
variety, and for $X\subset\R^m$ subanalytic and $W\subset\R^m$ a
semi-algebraic set.

Our principal motivation for Definition~\ref{def:xw} is the following
direct consequence (cf. Theorem~\ref{thm:Rre-main}).

\begin{Lem}\label{lem:alg-part-trans}
  Let $W\subset\R^m$ be a connected positive-dimensional semialgebraic
  set and $A\subset\R^m$. Then $A(W)\subset A^\alg$.
\end{Lem}

We record some simple consequences of Definition~\ref{def:xw}.
\begin{Lem}\label{lem:alg-part-rules}
  Let $A,B,W\subset\C^m$. Then
  \begin{equation}
    A(W)\cup B(W)\subset(A\cup B)(W).
  \end{equation}
  If $A\subset B$ is relatively open then
  \begin{equation}
    B(W)\cap A = A(W).
  \end{equation}
\end{Lem}

\subsection{Projections from admissible graphs}
\label{sec:admissible-graph}

Let $\Omega_x\subset\C^m$ and $\Omega_y\subset\C^n$ be domains and set
$\Omega:=\Omega_x\times\Omega_y\subset\C^{n+m}$. We denote by
$\pi:\Omega\to\Omega_x$ the projection map.

Let $U\subset\Omega_x$ be an open subset and $\psi:U\to\Omega_y$
a function, and denote its graph by
\begin{equation}
  \Gamma_\psi\subset\Omega, \qquad \Gamma_\psi := \{ (x,\psi(x)) : x\in U\}.
\end{equation}
We denote by $\tilde\psi:U\to\Gamma_\psi$ the map $x\to(x,\psi(x))$.

\begin{Def}\label{def:admissible}
  We say that $\psi$ is \emph{admissible} if $\Gamma=\Gamma_\psi$ is
  relatively compact in $\Omega$, and if there exists an analytic
  subset $X_\Gamma$ of $\Omega$ which agrees with $\Gamma$ over $U$,
  i.e. $X_\Gamma\cap\pi^{-1}(U)=\Gamma$.
\end{Def}

\begin{Ex}\label{ex:admissible-division}
  Let $\Omega_x\subset\C^2$ be a domain such that the unit ball
  $B_2\subset\C^2$ is a relatively compact subset of $\Omega_x$ and
  $\Omega_y\subset\C$ be a domain such that the unit ball
  $B_1\subset\C$ is a relatively compact subset of $\Omega_y$. Let
  \begin{equation}
    U:=\{(x_1,x_2)\in B_2:|x_1|<|x_2|\}
  \end{equation}
  and define $\psi:U\to\Omega_y$ by $(x_1,x_2)\to(x_1/x_2)$. Then
  $\psi$ is an admissible projection, as its graph over $U$ agrees
  with the analytic subset $X_\Gamma\subset\Omega$ given by
  $yx_2=x_1$. This example is essentially the only case that we shall
  require in the sequel.
\end{Ex}

In Theorem~\ref{thm:exploring-projection} we prove an analog of the
Wilkie conjecture for images of holomorphic Pfaffian varieties under
admissible projections. In~\secref{sec:definable} we study
$\R^\RE$-definable sets, and show that (for the purpose of counting
rational points) one can reduce any definable set to the image of such
an admissible projection (see
Proposition~\ref{prop:D-equation-to-proj}).

\subsection{Rational points on admissible projections}

We fix an admissible map $\psi:U\to\Omega_y$ and denote
$\Gamma:=\Gamma_\psi$. In this section we will consider a fixed
holomorphic-Pfaffian variety $X$, and compute asymptotics for the
number of rational points with respect to the height $H$. Therefore in
our asymptotic notations we allow ours constants to depends $X$ and
$\Gamma$, as well as $[\cF:\Q]$. We note however that the estimates do
not depend on $\cF$ itself. The following is our main result in this
section.

\begin{Thm}\label{thm:exploring-projection}
  Let $X\subset\Omega$ be a holomorphic-Pfaffian variety defined by
  $S$ sub-Pfaffian functions with complexity bounded by
  $(2n+2m+2,r,s,\ell,\alpha,\beta)$. Suppose $X\subset X_\Gamma$ and
  set
  \begin{equation}
    Y := \pi(X\cap\Gamma) \subset \Omega_x.
  \end{equation}
  Then for
  \begin{equation}
    \kappa = 3^{m}\left(m+n+2S(r+1)+\ell+1\right)^{3m}
  \end{equation}
  and any $H\in\N$ there exists algebraic varieties $V_0,\ldots,V_m$
  such that $V_j$ has pure dimension $j$ and degree
  $O((\log H)^\kappa)$, and
  \begin{equation}
    Y^\size(\cF,H) \subset Y(V_0)\cup\cdots\cup Y(V_m).
  \end{equation}
\end{Thm}

We will need the following basic lemma.

\begin{Lem}\label{lem:W-degrees}
  Let $W\subset\C^m$ be a pure dimensional algebraic variety of degree
  $d$. Then:
  \begin{enumerate}
  \item $W$ is set-theoretically cut out by a set of at most $m+1$
    polynomials, each of degree at most $d$.
  \item The exists a polynomial $Q$ of degree at most $d$
    vanishing identically on $\Sing W$ but not on $W$.
  \end{enumerate}
\end{Lem}
\begin{proof}
  Assume first that $W$ is a hypersurface. Then $W=\{P=0\}$ where $P$
  is square free and $\deg P=d$, giving the first statement. Any
  non-vanishing derivative of $P$ satisfies the second statement.

  For the general case, any linear projection $\pi_L(W)$ to a
  $(\dim W+1)$-dimensional space $L$ has degree bounded by $d$ and
  hence is cut out by a polynomial $P_L$ of degree at most $d$. It
  is easy to see that the polynomials $P'_L:=\pi_L^* P_L$ over all $L$
  cut out $W$ set-theoretically. To see the $m+1$ of them suffice,
  choose a generic $L_1$ such that $P_{L_1}$ does not vanish
  identically on $\C^m$ and write $W_1=\{P_{L_1}=0\}$. If
  $C\subset W_1$ is a component of $W_1$ which is not a component of
  $W$ then, since $\{P'_L\}$ cut out $W$ set theoretically, a generic
  $P'_L$ does not vanish on $C$. Choose a generic $L_2$ such that
  $P_{L_2}$ does not vanish identically on any such component $C$, and
  write $W_2=W_1\cap\{P_{L_2}=0\}$. Continuing in the same manner we
  see that any component $C\subset W_k$ which is not a component of
  $W$ has dimension at most $m-k$. In particular $W_{m+1}=W$.

  For the second statement choose generic $L$ such that the projection
  of $W$ to $L$ is generically one-to-one. Then $\pi_L(W)$ is reduced
  and its singular point are the points of local multiplicity greater
  than one. Since linear projections do not decrease local
  multiplicity we have $\pi_L(\Sing W)\subset \Sing \pi_L(W)$. If
  $Q_L$ is the polynomial constructed for $\pi_L(W)$ then
  $\pi_L^* Q_L$ satisfies the second statement for $W$.
\end{proof}

We begin the proof of Theorem~\ref{thm:exploring-projection} with the
following proposition.

\begin{Prop}\label{prop:exploring-projection}
  Let $X\subset\Omega$ be a holomorphic-Pfaffian variety defined by
  $S$ sub-Pfaffian function with complexity bounded by
  $(2n+2m+2,r,s,\ell,\alpha,\beta)$. Suppose $X\subset X_\Gamma$ and
  set
  \begin{equation}
    Y := \pi(X\cap\Gamma) \subset \Omega_x.
  \end{equation}
  Let $W\subset\C^m$ be an algebraic variety of pure dimension $k$ and
  degree $d$. Then for
  \begin{multline}
    \lambda(k) = \lambda(n,m,r,S,\ell,k) := (n+m+2S(r+1)+\ell)(n+m-k+1)
    \\ +(n+m)\theta(n+m,k-1,2S(r+1),\ell)+1
  \end{multline}
  and for any $H\in\N$ there exists an algebraic variety
  $V\subset W$ of pure dimension $k-1$ and degree
  $O(d^{\l(k)}(\log H)^{k-1})$ such that
  \begin{equation}
    (Y\cap W)^\size(\cF,H) \subset Y(W)\cup V.
  \end{equation}
\end{Prop}
\begin{proof}
  Set
  \begin{equation}
    Z:=(X\cap(W\times\Omega_y))^{<k}.
  \end{equation}
  Let $q\in Y\cap W$ and suppose that $q\not\in\Sing W$ and
  $q\not\in Y(W)$. Then the germ $W_q$ of $W$ at $q$ is smooth
  $k$-dimensional and not contained in $Y$. Equivalently, its image
  $\tilde\psi(W_q)$ is the germ of a smooth $k$-dimensional analytic set
  at $\tilde\psi(q)$ which is not contained in $X$. Since we assume
  $X\subset\Gamma$ around $\tilde\psi(q)$ we deduce that the dimension
  of
  \begin{equation}
    X\cap(W_q\times\Omega_y) = X\cap\Gamma\cap(W_q\times\Omega_y)=X\cap\tilde\psi(W_q)
  \end{equation}
  at $\tilde\psi(q)$ is strictly smaller than $k$, i.e.
  $\tilde\psi(q)\in Z$. In conclusion,
  \begin{equation}\label{eq:exploring-proj-1}
    Y\cap W\subset Y(W)\cup \Sing W \cup \pi(Z).
  \end{equation}

  By Lemma~\ref{lem:W-degrees} there exists a hypersurface
  $\cH_0\subset\C^m$ of degree at most $d$ containing $\Sing W$ and
  not containing $W$. Also, the holomorphic-Pfaffian variety
  $X\cap (W\times\Omega_y)$ is cut out by the equations for $X$ and a
  set of additional polynomials equations (in $\vx$) of degrees
  bounded by $d$.
  
  Let $p\in\bar\Gamma$. Since $\Gamma\subset\Omega$ is relatively
  compact there exists a Euclidean ball $B_p\subset\Omega$ around $p$
  of radius $\Omega(1)$. Slightly shrinking $B_p$ if necessary we may
  also assume $B_p^{1/3}\subset\Omega$. By
  Corollary~\ref{cor:pfaff-polydisc} there exists an analytic set
  $Z'\subset B_p$ of pure dimension $k-1$ satisfying $Z\subset Z'$, and
  a Weierstrass polydisc $\Delta$ for $Z$ such that
  \begin{enumerate}
  \item $B^\eta\subset\Delta\subset B$ where $\eta=O(d^\theta)$
    and $\theta = \theta(n+m,k-1,2S(r+1),\ell)$.
  \item $e(Z',\Delta)=O(d^{n+m+2S(r+1)+\ell})$.
  \end{enumerate}
  Note that in the estimate above we use $S$, the number of
  sub-Pfaffian holomorphic equations for $X$, and do not count the
  additional equations used to define $W$. This is permissible in
  light of Remark~\ref{rem:pfaff-polydisc-S} and improves the
  asymptotics.

  Assume first that $W$ is irreducible. Then one can choose a subset
  of $k$ coordinates on $\C^m$, say $\vf=(x_1,\ldots,x_k)$ such that
  $\vf:W\to\C^k$ is dominant and in particular no non-zero polynomial
  in $\vf$ vanishes identically on $W$. We apply
  Proposition~\ref{prop:hypersurface-select} to $Z'$ and $\vf$. We
  conclude that
  \begin{equation}\label{eq:exploring-proj-2}
    [\pi(Z'\cap B_p^{2\eta})]^\size(\cF,H)\subset\{P_p(\vf)=0\}
  \end{equation}
  for some non-zero polynomial $P_p(\vf)$ of degree $\tilde d$, where
  \begin{equation}
    \begin{split}
      \tilde d &= O(e(Z',\Delta)^{n+m-k+1} (\log H)^{k-1}) \\ 
               &= O(d^{(n+m+2S(r+1)+\ell)(n+m-k+1)} (\log H)^{k-1}).
    \end{split}
  \end{equation}
  Finally, since the ball $B_p^{2\eta}$ has radius
  $\Omega(\eta^{-1})$, one can choose a covering of $\bar\Gamma$ by
  $O(\eta^{n+m})$ such balls. We let $\cH$ be the union of the
  corresponding hyperplanes $\{P_p(\vf)=0\}$, and take $V'=W\cap\cH$
  and $V=V'\cup(W\cap \cH_0)$. Then the degree estimates follow from
  the Bezout theorem and the statement follows
  from~\eqref{eq:exploring-proj-1},~\eqref{eq:exploring-proj-2} and
  the choice of $\cH_0$.

  If $W$ is reducible with components $W_i$ then we may repeat the
  construction above for each $W_i$ separately, and take $V'$ to be
  the union of the resulting $V_i'$ and $V=V'\cup(W\cap \cH_0)$ as
  before. The degree estimates in this case are only improved.
\end{proof}

\begin{proof}[Proof of Theorem~\ref{thm:exploring-projection}.]
  Define $\kappa(m),\ldots,\kappa(0)$ by
  \begin{align}
    \kappa(m)&=0 & \kappa(k) &= k-1+\lambda(k)\kappa(k+1)
  \end{align}
  Apply Proposition~\ref{prop:exploring-projection} with $W=V_m:=\C^m$
  to obtain an algebraic variety $V_{m-1}\subset\C^m$ of pure
  dimension $m-1$ and degree $O((\log H)^{\kappa(m-1)})$ such that
  \begin{equation}
    Y^\size(\cF,H) \subset Y(V_m)\cup V_{m-1}.
  \end{equation}
  Apply Proposition~\ref{prop:exploring-projection} again with
  $W=V_{m-1}$ to obtain an algebraic variety $V_{m-2}\subset\C^m$ of
  pure dimension $m-2$ and degree $O((\log H)^{\kappa(m-2)})$ such
  that
  \begin{equation}
    (Y\cap V_{m-1})^\size(\cF,H) \subset Y(V_{m-1})\cup V_{m-2}.
  \end{equation}
  Repeating similarly we obtain an algebraic variety $V_k\subset\C^m$
  of pure dimension $k$ and degree $O((\log H)^{\kappa(k)})$ such
  that
  \begin{equation}\label{eq:W_k-def}
    (Y\cap V_k)^\size(\cF,H) \subset Y(V_k)\cup V_{k-1}
  \end{equation}
  where $V_{-1}=\emptyset$. Finally using $Y(V_0)=Y\cap V_0$
  and~\eqref{eq:W_k-def} gives
  \begin{equation}
    Y^\size(\cF,H) \subset Y(V_0)\cup\cdots\cup Y(V_m).
  \end{equation}
  An easy estimate on $\kappa(k)$ finishes the proof:
  as $\kappa(k)+1\le \lambda(k)\left(\kappa(k+1)+1\right)$, we have 
  \begin{equation*}
  \kappa(0)\le\prod_{k=0}^{m-1}\lambda(k)\le 3^{m}\left(m+n+2S(r+1)+\ell+1\right)^{3m}.
  \end{equation*}
\end{proof}

\section{Definable sets in $\R^\RE$ and the language $L^D_\RE$}
\label{sec:definable}

Let $I=[-1,1]$. For $m\in\N$ we let $\R^\RE\{X_1,\ldots,X_m\}$ denote
the ring of power series $f\in\R[[X_1,\ldots,X_m]]$ such that
\begin{enumerate}
\item $f$ converges in a neighborhood of $I^m$.
\item For every point $p\in I^m$ there is a polydisc $\Delta_p$ around
  $p$ such that $f$ converges in $\Delta_p$ to a holomorphic function,
  whose graph (in $\Delta_p$) is definable in $\R^\RE$.
\end{enumerate}
We remark that \cite{vdd:Rre} requires \emph{strong definability} in
item 2 above, but this is in fact equivalent to definability by the
main result of \cite{vdd:Rre}. By
Proposition~\ref{prop:Rre-subPfaffian} and the compactness of $I^m$,
for every every function $f\in\R^\RE\{X_1,\ldots,X_m\}$ there exists a
complex neighborhood $\Omega_f$ of $I^m$ such that $f$ is holomorphic
and sub-Pfaffian in $\Omega_f$.

We recall the language $L^D_\RE$ of \cite{vdd:Rre}. There is a countable set
of variables $\{X_1,X_2,\ldots\}$, a relation symbol $<$ and a binary
operation symbol $D$, and an $m$-ary operation symbol $f$ for every
$f\in\R^\RE\{X_1,\ldots,X_m\}$ satisfying $f(I^m)\subset I$. We view $I$
as an $L^D_\RE$-structure by interpreting $<$ and $f$ in the obvious
way and interpreting $D$ as \emph{restricted division}, namely
\begin{equation}
  D(x,y) =
  \begin{cases}
    x/y & |x|\le|y| \text{ and } y\neq0 \\
    0 & \text{otherwise.}
  \end{cases}
\end{equation}
We denote by $L_\RE$ the language obtained from $L_\RE^D$ by omitting
$D$.

For every $L_\RE^D$-term $t(X_1,\ldots,X_m)$ we have an associated map
$t:I^m\to I$ which we denote $x\to t(x)$.
\begin{Lem}\label{lem:Lre-term}
  Let $t(X_1,\ldots,X_m)$ be an $L_\RE$ term. Then there is a complex
  neighborhood $\Omega$ of $I^m$ such that $t$ corresponds to a
  holomorphic sub-Pfaffian function in $\Omega$.
\end{Lem}
\begin{proof}
  We prove the claim by induction: if $t=X_i$ then the claim is
  obvious. Suppose $t=f(T_1,\ldots,T_j)$ where
  $f\in\R^\RE\{Y_1,\ldots,Y_j\}$ and the claim is proved for
  $T_1,\ldots,T_j$. Then there exists a complex neighborhood
  $\Omega_f$ of $I^j$ such that $f$ is holomorphic and sub-Pfaffian in
  $\Omega_f$, and complex neighborhoods $\Omega_i$
  of $I^m$ such that $T_i$ corresponds to a holomorphic sub-Pfaffian
  function in $\Omega_i$. Since $T_i(I^n)\subset I$ we may, shrinking
  $\Omega_i$ if necessary, assume that $(T_1,\ldots,T_j)$ maps
  $\Omega:=\Omega_1\times\cdots\times\Omega_j$ to $\Omega_f$. Then $t$
  corresponds to a holomorphic function in $\Omega$, and since
  sub-Pfaffian functions are closed under composition it is also
  sub-Pfaffian.
\end{proof}

For an $L_\RE^D$-formula $\phi(X_1,\ldots,X_m)$ we write $\phi(I^m)$
for the set of points $x\in I^m$ satisfying $\phi$. If $A\subset I^m$
we write $\phi(A):=\phi(I^m)\cap A$. We will use the following key
result of \cite{vdd:Rre}.
\begin{Thm}\label{thm:denef-vdd}
  $I$ has elimination of quantifiers in $L^D_\RE$. 
\end{Thm}

\subsection{Admissible formulas}

Let $U\subset\mathring{I}^m$ be an open subset. We define the notion
of an $L_\RE^D$-term \emph{admissible} in $U$ by recursion as follows:
a variable $X_j$ is always admissible in $U$; a term
$f(t_1,\ldots,t_m)$ is admissible in $U$ if and only if the terms
$t_1,\ldots,t_m$ are admissible in $U$; and a term $D(t_1,t_2)$ is
admissible in $U$ if $t_1,t_2$ are admissible in $U$ and if
\begin{equation}
  |t_1(x)| \le |t_2(x)| \text{ and } t_2(x)\neq0
\end{equation}
for every $x\in U$. An easy induction gives the following.
\begin{Lem}\label{lem:admissible-term-analytic}
  If $t$ is admissible in $U$ then the map $t:U\to I$ is real
  analytic.
\end{Lem}

We will say that an $L_\RE^D$-formula $\phi$ is admissible in $U$ if
all terms appearing in $\phi$ are admissible in $U$. The following
proposition shows that when considering definable subsets of $I$ one
can essentially reduce to admissible formulas. The proof is
identical to that of \cite[Proposition~20]{me:interpolation}.

\begin{Prop}\label{prop:admissible-decomp}
  Let $U\subset\mathring I^m$ be an open subset and
  $\phi(X_1,\ldots,X_m)$ a quantifier-free $L_\RE^D$-formula. There
  exist open subsets $U_1,\ldots,U_k\subset U$ and quantifier-free
  $L_\RE^D$-formulas $\phi_1,\ldots,\phi_k$ such that $\phi_j$ is
  admissible in $U_j$ and
  \begin{equation}
    \phi(U) = \bigcup_{j=1}^k \phi_j(U_j).
  \end{equation}
\end{Prop}

\subsection{Basic formulas and equations}

We say that $\phi$ is a \emph{basic $D$-formula} if it has
the form
\begin{equation}
  \big(\land_{j=1}^k t_j(X_1,\ldots,X_m)=0\big)\land\big(\land_{j=1}^{k'} s_j(X_1,\ldots,X_m)>0\big)
\end{equation}
where $t_j,s_j$ are $L_\RE^D$-terms. It is easy to check the
following.

\begin{Lem}\label{lem:basic-disjunction}
  Every quantifier free $L_\RE^D$-formula $\phi$ is equivalent in the structure
  $I$ to a finite disjunction of basic formulas. If $\phi$ is
  $U$-admissible then so are the basic formulas in the disjunction.
\end{Lem}

We say that $\phi$ is a \emph{basic $D$-equation} if $k'=0$, i.e. if
it involves only equalities. If $\phi$ is a basic $D$-formula we
denote by $\tilde\phi$ the basic $D$-equation obtained by removing all
inequalities.

Let $\phi$ be a $U$-admissible basic $D$-formula for some
$U\subset I^m$. Then $\tilde\phi$ is $U$-admissible as well. Moreover
since all the terms $s_j$ evaluate to continuous functions in $U$ the
strict inequalities of $\phi$ are open in $U$ and we have the
following.

\begin{Lem}\label{lem:phi-tilde-rel-open}
  Suppose $\phi$ is $U$-admissible basic $D$-formula. Then $\phi(U)$
  is relatively open in $\tilde\phi(U)$.
\end{Lem}

The set defined by an admissible $D$-equation can be described in
terms of admissible projections in the sense
of~\secref{sec:admissible-graph}.

\begin{Prop}\label{prop:D-equation-to-proj}
  Let $U\subset\mathring I^m$ and $\phi$ be a $U$-admissible $D$-equation,
  \begin{equation}
    \phi = (t_1=0)\land\cdots\land(t_k=0).
  \end{equation}
  In the notations of~\secref{sec:admissible-graph}, there exist
  \begin{enumerate}
  \item Complex domains $\Omega_x\subset\C^m$ and
    $\Omega_y\subset\C^N$ with $N\in\N$ and $I^m\subset\Omega_x$.
  \item An open complex neighborhood $U\subset U_\C\subset \Omega_x$.
  \item An analytic map $\psi:U_\C\to\Omega_y$.
  \item An analytic set $X\subset \Omega$.
  \end{enumerate}
  such that $\psi$ is admissible, the sets $X,X_\Gamma\subset\Omega$
  are sub-Pfaffian, and $Y:=\pi(X\cap\Gamma_\psi)$ satisfies
  $Y_\R=\phi(U)$.
\end{Prop}
\begin{proof}
  The proof is essentially the same as the proof in
  \cite[Proposition~23]{me:interpolation} for the language
  $L^D_{\mathrm{an}}$. We note that in \cite{me:interpolation} this
  proposition is proved with an extra set of parameters $\Lambda$, and
  in the current context one can take $\Lambda$ to be a singleton. To
  see that the sets $X,X_\Gamma$ are also sub-Pfaffian use
  Lemma~\ref{lem:Lre-term}.
\end{proof}

\subsection{Estimate for $L_\RE^D$-definable sets}

\begin{Prop}\label{prop:Lre-main}
  Let $A\subset\mathring{I}^m$ be an $L_\RE^D$-definable set. There
  exist integers $\kappa=\kappa(A)$ and $N=N(A,[\cF:\Q])$ with the
  following property: for any $H\in\N$ there exist at most
  $\beta:=N(\log H)^\kappa$ smooth connected semialgebraic subsets
  $S_\alpha\subset\R^m$ with complexity $(m,\beta,\beta)$ such that
  \begin{equation}
    A(\cF,H)\subset \bigcup_\alpha A(S_\alpha).
  \end{equation}
\end{Prop}
\begin{proof}
  According to~\eqref{eq:H-vs-Hsize}, for any $\alpha\in\cF$ we have
  $H^\size(\alpha)\le H(\alpha)^t$ where $t:=[\cF:\Q]$. Thus
  \begin{equation}
    A(\cF,H) \subset A^\size(\cF,H^t).
  \end{equation}
  for every $H\in\N$. Therefore up to minor rescaling it will suffice
  to prove the claim with $H(\cdot)$ replaced by $H^\size(\cdot)$.

  By Theorem~\ref{thm:denef-vdd} we may write
  $A=\phi(\mathring I^m)$ for some quantifier-free
  $L^D_\RE$-formula $\phi$. By
  Proposition~\ref{prop:admissible-decomp} and
  Lemma~\ref{lem:basic-disjunction} we may write
  \begin{equation}
    A = \bigcup_{j=1}^k \bigcup_{i=1}^{n_j} \phi_{ji}(U_j)
  \end{equation}
  where $\phi_{ji}$ is a $U_j$-admissible basic $D$-formula. By the
  first part of Lemma~\ref{lem:alg-part-rules} it is clear that it
  will suffice to prove the claim with $A$ replaced by each
  $\phi_{ij}(U_j)$. We thus assume without loss of generality that
  $\phi$ is already a $U$-admissible basic $D$-formula and prove the
  claim for $A=\phi(U)$.

  Recall that $\tilde\phi$ is a $U$-admissible $D$-equation. We write
  $B=\tilde\phi(\mathring I^m)$. Applying
  Proposition~\ref{prop:D-equation-to-proj} to $\tilde\phi$ and using
  Theorem~\ref{thm:exploring-projection} we construct a locally
  analytic set $Y\subset\C^m$ such that $Y_\R=B$, and algebraic
  varieties $V_0,\ldots,V_m$ such that $V_j$ has pure dimension $j$
  and degree $O((\log H)^{\kappa'})$ such that
  \begin{equation}
    Y^\size(\cF,H) \subset Y(V_0)\cup\cdots\cup Y(V_m).
  \end{equation}
  By~\eqref{eq:alg-part-R} and $Y_\R=B$ we have
  \begin{equation}
    B(\cF,H)^\size\subset B(\cV_0)\cup\cdots\cup B(\cV_m), \qquad \cV_j:=(V_j)_\R
  \end{equation}
  Recall that $A$ is relatively open in $B$ by
  Lemma~\ref{lem:phi-tilde-rel-open}. Then
  \begin{equation}
    \begin{split}
      A^\size(\cF,H)&= A\cap(B^\size(\cF,H))\subset A\cap
      \bigcup_{j=0}^m B(\cV_j) = \bigcup_{j=0}^m ( A\cap B(\cV_j) ) \\ &= \bigcup_{j=0}^m A(\cV_j)
    \end{split}
  \end{equation}
  where the last equality is given by the second part of
  Lemma~\ref{lem:alg-part-rules}.

  If we write each $\cV_j$ as a union of connected smooth strata
  $\cV_j=\cup_l S_{j,l}$ then we have
  \begin{equation}
    A^\size(\cF,H)\subset \bigcup_j A(\cV_j) \subset \bigcup_{j,l} A(S_{j,l}).
  \end{equation}
  It remains to estimate the number and complexity of the strata.
  Recall from Lemma~\ref{lem:W-degrees} that each $V_j$ is cut out by
  $m+1$ complex equations of degree at most $\deg V_j$, which is
  bounded by $\beta_1:=O((\log H)^{\kappa'})$. The real-part $\cV_j$
  is cut out by the same equations and additional linear equations for
  the vanishing of all imaginary parts, and thus has complexity
  $(m,3m+2,\beta_1)$. By \cite[Theorem~2]{gv:strata} one can decompose
  $\cV_j$ into a union of $\beta_2:=\beta_1^{(2^{O(m)})}$ smooth (but
  not necessarily connected) semialgebraic sets of complexity
  $(m,\beta_2,\beta_2)$. Finally, by \cite[Theorem~16.13]{basu:book}
  each such semialgebraic set can be decomposed into its connected
  components, with the number of connected components bounded by
  $\beta_3=\beta_2^{(m^4)}$ and their complexity bounded by
  $(m,\beta_3,\beta_3)$ (a better estimate for the number of connected
  components follows from Theorem~\ref{thm:pfaffian-complexity}).
\end{proof}

\subsection{Estimate for $\R^\RE$-definable sets}

Consider the map $\tau:\R\to \mathring I$ given by
$x\to x/\sqrt{1+x^2}$. This map is a bijection between $\R$ and
$\mathring I$, with the inverse $\tau^{-1}:\mathring I\to\R$ given by
$y\to y/\sqrt{1-y^2}$. A straightforward verification
\cite[4.6]{vdd:Rre} shows that the $\tau$-images of the basic
relations of $\R^\RE$ are $L_\RE^D$-definable in $I$, and it follows
that every $\R^\RE$-definable set $A\subset\R^n$ has an
$L_\RE^D$-definable $\tau$-image $\tau(A)\subset\mathring{I}^m$.

\begin{Lem}\label{lem:Rre-Lre}
  Suppose $A\subset\mathring{I}^m$ is $\R^\RE$-definable. Then $A$ is
  $L_\RE^D$-definable.
\end{Lem}
\begin{proof}
  Let $I':=\tau I=[-\tfrac1{\sqrt2},\tfrac1{\sqrt2}]$. By the above
  $\tau(A)\subset \mathring I'$ is $L_\RE^D$-definable. The
  restriction $\tau^{-1}\rest{I'}:I'\to[-1,1]$ is definable in $I$ by
  the $L_\RE^D$-formula
  \begin{equation}
    \phi(y,x) := \exists z : z\ge0, z^2=1-y^2, x=D(y,z). 
  \end{equation}
  Then $A=\tau^{-1}(\tau A)$ is $L_\RE^D$-definable as well.
\end{proof}

The following Theorem is the general form of our main result.

\begin{Thm}\label{thm:Rre-main}
  Let $A\subset\R^m$ be an $\R^\RE$-definable set. There exist
  integers $\kappa=\kappa(A)$ and $N=N(A,[\cF:\Q])$ with the following
  property: for any $H\in\N$ there exist at most
  $\beta:=N(\log H)^\kappa$ smooth connected semialgebraic subsets
  $S_\alpha\subset\R^m$ with complexity $(m,\beta,\beta)$ such that
  \begin{equation}
    A(\cF,H)\subset \bigcup_\alpha A(S_\alpha).
  \end{equation}
\end{Thm}
\begin{proof}
  For $A\subset\mathring I^m$ the claim follows by
  Lemma~\ref{lem:Rre-Lre} and Proposition~\ref{prop:Lre-main}. For the
  general case, note that the definable transformations $x_i\to 1/x_i$
  and $x_i\to x_i+1$ do not affect the heights of the points of $A$ by
  more than a constant factor, and their pullbacks preserve the
  smoothness of $S_\alpha$ do not affect the degrees of $S_\alpha$ by
  more than a constant factor. It is easy to construct a finite set
  $\{T_j\}$ where each $T_j:\R^m\to\R^m$ is a finite composition of
  the transformations above such that for any $A\subset\R^m$,
  \begin{equation}
    A = \bigcup_j T_j^{-1} ( T_j(A)\cap\mathring I^m ).
  \end{equation}
  The claim now follows from the case $A\subset\mathring I^m$ already
  considered.
\end{proof}

\begin{Rem}
  The reader may note that the strata $S_\alpha$ in
  Theorem~\ref{thm:Rre-main} play essentially the same role as the
  ``basic blocks'' introduced in \cite{pila:algebraic-points}. They
  will be used in a similar way for the proof of
  Theorem~\ref{thm:main-k}.
\end{Rem}

\subsection{Proofs of the main theorems}

In this section we prove our two main results Theorem~\ref{thm:main-F}
and Theorem~\ref{thm:main-k}. We start with Theorem~\ref{thm:main-F}
which is essentially an immediate consequence of
Theorem~\ref{thm:Rre-main}.

\begin{proof}[Proof of Theorem~\ref{thm:main-F}.]
  Let $H\in\N$. Consider the collection of $\beta$ smooth connected
  semialgebraic sets $S_\alpha$ obtained from
  Theorem~\ref{thm:Rre-main}, such that
  \begin{equation}
    A(\cF,H)\subset \bigcup_\alpha A(S_\alpha).
  \end{equation}
  By Lemma~\ref{lem:alg-part-trans} we see that $A^\trans(\cF,H)$ is
  contained in the union of the zero-dimensional strata $S_\alpha$.
  Then $\#(A^\trans)(\cF,H)$ is bounded by the number of strata, which
  is of the required form by Theorem~\ref{thm:Rre-main}.
\end{proof}

We now pass to the proof of Theorem~\ref{thm:main-k}, which is a
direct adaptation of the approach of \cite{pila:algebraic-points}
(with slightly more effort needed to obtain the necessary degree
estimates). We introduce some additional notation for this purpose.
Let $\cP_{\le k}:=\R^{k+1}\setminus\{{\mathbf0}\}$ and
$\cP_k:=\{\vc\in\cP_{\le k}:c_k\neq0\}$. For $\vc\in\cP_{\le k}$ let
$P_\vc\in\R[x]$ denote the polynomial
\begin{equation}
  P_\vc(X) := \sum_{j=0}^k c_j X^j.
\end{equation}
We let $D_k\subset\cP_k$ denote the \emph{discriminant}, i.e. the set
of $\vc\in\cP_k$ such that $P_\vc$ has a (possibly complex) double
zero. Then $D_k$ is an algebraic subset and its complement
$\tilde\cP_k:=\cP_k\setminus D_k$ is an open semialgebraic set. We let
$Z_k\subset\tilde\cP_k\times\R$ be the set
\begin{equation}
  Z_k := \{(\vc,x)\in\tilde\cP_k\times\R : P_\vc(x)=0\}.
\end{equation}
\begin{Lem}\label{lem:branch-select}
  Let $\{U_\chi\}$ denote the connected components of $\tilde\cP_k$.
  For each $\chi$ there is a tuple of at most $k$ real-analytic
  algebraic functions $\phi_{\chi,j}:U_\chi\to\R$ such that
  \begin{equation}
    Z_k = \bigcup_{\chi,j} \Gamma_{\phi_{\chi,j}}
  \end{equation}
  where $\Gamma_\phi$ denotes the graph of $\phi$.
\end{Lem}
\begin{proof}
  Since $\tilde\cP_k$ is the complement of the discriminant, the
  number of zeros of $P_\vc$ for $\vc\in U_\chi$ is some constant
  $n_\chi\le k$, and since $P_\vc$ has no double zeros one can
  uniquely choose the branch $\phi_{\chi,j}$ for $j=1,\ldots,n_\chi$
  to be the $j$-th root of $P_\vc$ in increasing order. The branches
  thus constructed are Nash functions, i.e. real analytic and
  algebraic, on $U_\chi$. Their graphs cover $Z_k\cap(U_\chi\times\R)$
  by construction.
\end{proof}

Following \cite{pila:algebraic-points} we introduce the following
height function. For an algebraic number $\alpha\in\Qa$ we define
\begin{equation}
  H_k^\poly(\alpha) = \min\{H(\vc) : \vc\in\cP_{\le k}(\Q), \quad P_\vc(\alpha)=0\}
\end{equation}
and $H(\alpha)=\infty$ if $[\Q(\alpha):\Q]>k$. Then whenever
$[\Q(\alpha):\Q]\le k$ we have \cite[5.1]{pila:algebraic-points}
\begin{equation}\label{eq:H-vs-Hpoly}
  H_k^\poly(\alpha) \le 2^k H(\alpha)^k.
\end{equation}
For a set $A\subset\R^m$ we define $A^\poly(k,H)$ in analogy
with $A(k,H)$ replacing $H(\cdot)$ by $H^\poly(\cdot)$.

\begin{proof}[Proof of Theorem~\ref{thm:main-k}, adapted from \protect{\cite{pila:algebraic-points}}.]
  As in the proof of Theorem~\ref{thm:main-F},
  from~\eqref{eq:H-vs-Hpoly} we deduce that up to minor rescaling it
  will suffice to prove the claim with $H(\cdot)$ replaced by
  $H^\poly(\cdot)$. Let $\vk=(k_1,\ldots,k_m)\in\N^m$ and denote
  \begin{align}
    A(\vk) &:= \{\vx\in A:[\Q(x_1):\Q]=k_1,\ldots,[\Q(x_m):\Q]=k_m\}, \\
    A(\vk,H) &:= \{\vx\in A(\vk):H(\vx)\le H\}.  
  \end{align}
  Then
  \begin{equation}
    A(k,H) = \bigcup_{\vk:1\le k_j\le k} A(\vk,H)
  \end{equation}
  and it will suffice to prove the claim for fixed $\vk$ in place of
  $k$.

  Denote $\tilde\cP^\vk:=\prod_{j=1}^m\tilde\cP^{k_j}$ and let
  $Z_\vk\subset \tilde\cP_\vk\times\R^m$ be the set
  \begin{multline}
    Z_\vk := \{(\vc^1,\ldots,\vc^m,x_1,\ldots,x_m)\in \tilde\cP^k\times\R^m :\\
    P_{\vc^1}(x_1)=\cdots=P_{\vc^m}(x_m)=0\}.
  \end{multline}
  Let $\{U_\vchi\}$ denote the connected components of
  $\tilde\cP_\vk$. For each $\vchi$ there is a tuple of at most $k^m$
  real-analytic algebraic maps $\phi_{\vchi,j}:U_\vchi\to\R^m$ such
  that
  \begin{equation}\label{eq:Z-vk-branches}
    Z_\vk = \bigcup_{\vchi,j} \Gamma_{\phi_{\vchi,j}}
  \end{equation}
  where $\Gamma_\phi$ denotes the graph of $\phi$. Indeed, $U_\vchi$
  are just direct products of the connected components of
  $\tilde\cP^{k_j}$, and the claim follows by taking the direct
  products of the functions constructed in
  Lemma~\ref{lem:branch-select}.

  For each component $U_\vchi$ and map $\phi_{\vchi,j}$ we define the set
  \begin{equation}\label{eq:A-vchi-def}
    A_{\vchi,j}\subset\tilde\cP^\vk, \quad A_{\vchi,j} := \{ p\in\tilde\cP^\vk : \phi_{\vchi,j}(p)\in A\}.
  \end{equation}
  Since $\phi_{\vchi,j}$ are semialgebraic, $A_{\vchi,j}$ are
  definable in $\R^\RE$. Consider the collection of at most
  $\beta=N(\log H)^\kappa$ smooth connected semialgebraic strata
  $S_{\vchi,j,\alpha}$ obtained from Theorem~\ref{thm:Rre-main}, such
  that
  \begin{equation}
    A_{\vchi,j}(\Q,H)\subset \bigcup_\alpha A_{\vchi,j}(S_{\vchi,j,\alpha}).
  \end{equation}

  Let $\vx\in A(\vk,H)$, and for $l=1,\ldots,m$ let
  $\vc^l\in\cP^{\le k_l}$ be a tuple satisfying $P_{\vc^l}(x_l)=0$ and
  $H(\vc^l)\le H$. Since $[\Q(x_l):\Q]=k_l$ we see that $P_{\vc^l}$ is
  (up to a scalar) the minimal polynomial of $x_l$, so it has degree
  $k_l$ and no multiple roots, i.e. $\vc^l\in\tilde\cP^{k_l}$. Write
  $p=(\vc^1,\ldots,\vc^m)\in\tilde\cP^\vk$. Then $(p,x)\in Z_\vk$ and
  by~\eqref{eq:Z-vk-branches} we have $\vx=\phi_{\vchi,j}(p)$ for some
  pair $(\vchi,j)$. By~\eqref{eq:A-vchi-def} we have
  $p\in A_{\vchi,j}(\Q,H)$. Choose one of the strata
  $S_{\vchi,j,\alpha}$ such that
  $p\in A_{\vchi,j}(S_{\vchi,j,\alpha})$ and denote it by $S(p)$ and
  its germ at $p$ by $S_p$. Then $S_p\subset A_{\vchi,j}$, and
  by~\eqref{eq:A-vchi-def} we have
  $Y_p:=\phi_{\vchi,j}(S_p)\subset A$. Note that $Y_p$ is a
  semialgebraic set containing $\phi_{\vchi,j}(p)=\vx$.

  Suppose $\vx\in A^\trans$. Then $\phi_{\vchi,j}$ is constant on
  $S_p$, for otherwise $Y_p$ would be a connected positive-dimensional
  semialgebraic set containing $\vx$. Since $S(p)$ is connected,
  nonsingular and analytic and $\phi_{\vchi,j}$ is real analytic it
  follows that $\phi_{\vchi,j}$ is constant (with value $\vx$) on the
  whole of $S(p)$.

  From the above it follows that there is an injective correspondence
  $p\to S(p)$ between the points $\vx\in A^\trans(\vk,H)$ and the
  strata $S_{\vchi,j,\alpha}$. The number of these strata for each
  $\vchi,j$ is at most $\beta$, and since the number of pairs $\vchi,j$
  is independent of $H$ we obtain a bound of the required form (but
  note that the exponent $\kappa$ now depends on the sets
  $A_{\vchi,j}$, i.e. on $k$ as well as $A$). This finishes the proof.
\end{proof}

\section{Concluding remarks}
\label{sec:concluding}

\subsection{Effectivity}
\label{sec:effectivity}
While we do not compute all explicit constants in this paper in the
interest of space, all of the estimates presented for
holomorphic-Pfaffian varieties and their admissible projections are
entirely effective in the complexity of the holomorphic-Pfaffian
varieties involved. We do give an effective estimate for the exponent
$\kappa$ in Theorem~\ref{thm:exploring-projection} as we believe this
may be of some interest in possible diophantine applications.

As a consequence of the above, our estimates for sets defined by
quantifier-free $L^D_\RE$-formulas can be made effective in terms of
the complexity of the formulas. We do not study the effectivity of the
quantifier-elimination result of \cite{vdd:Rre}, and we therefore
cannot make a direct statement about the effectivity of our main
results for general $\R^\RE$-definable sets. However, we believe that
by a combination of Pfaffian techniques and the proof of
\cite{vdd:Rre} it is possible to obtain an effective statement in this
context as well.

\subsection{Uniformity with respect to parameters}
Our approach is essentially uniform over definable families, and the
results could be developed in this additional generality in the same
way as this is done in \cite{me:interpolation} for the subanalytic
context. We avoided this extra generality in this paper in the
interest of preserving clarity; and also since in view
of~\secref{sec:effectivity} we believe a detailed analysis should give
estimates depending only on the complexity of the formulas involved,
and hence \emph{a-priori} uniform over families.

\subsection{Generalization to other structures}

We have focused in this paper on the structure $\R^\RE$ since it
contains the restricted form of Wilkie's original conjecture. It is of
course natural to make similar conjectures for sets definable in other
``tame'' geometric structures. We identify two possible categories for
such a generalization: structures generated by \emph{elliptic}
functions, for instance the Weierstrass $\wp$-function (for a fixed
lattice/s) and possibly higher-dimensional abelian functions; and
structures generated by \emph{modular} functions, for instance Klein's
$j$-invariant and possibly universal covering maps of more general
Shimura curves/varieties. In both cases one must of course restrict
the functions to a suitable fundamental domain to avoid periodicity
and obtain an O-minimal structure \cite{ps:theta-definability}. Both
categories are closely related to diophantine problems of unlikely
intersections: the former to the circle of problems around the
Manin-Mumford conjecture, and the latter to the circle of problems
around the Andr\'e-Oort conjecture \cite{pila:andre-oort}.

The elliptic case appears to be possibly amenable to our approach.
Namely, a surprising work of Macintyre \cite{macintyre:pfaffian} shows
that the real an imaginary parts of $\wp^{-1}$ (on an appropriate
domain) are real-Pfaffian functions, thus placing $\wp$ in the
holomorphic-Pfaffian category (in analogy with the function $e^z$
whose real an imaginary parts $e^x$ and $\sin x$, restricted to an
appropriate domain, are real-Pfaffian and generate $\R^\RE$). From the
model-theoretic side, the work of Bianconi \cite{bianconi:elliptic}
establishes a close analog of the results of \cite{vdd:Rre} for the
structures generated by real and imaginary parts of elliptic (or more
generally abelian) functions. It therefore appears likely that the
methods used in this paper could be carried over to the elliptic
category (at least for elliptic functions).

The modular category currently appears to be more challenging: we have
no reason to believe that the $j$-function is Pfaffian (or definable
from Pfaffian functions). However, we note that the $j$-function (as
well as other modular functions) does satisfy certain natural
non-Pfaffian systems of differential equations, and one may hope that
some progress in the analysis of such systems could provide a suitable
replacement for the Pfaffian category.

\bibliographystyle{plain} \bibliography{nrefs}

\end{document}